\documentclass[11pt,reqno]{amsart}
\usepackage{fullpage}
\usepackage[mathscr]{eucal}
\usepackage{mathrsfs}
\usepackage{amsfonts}
\usepackage{amsmath}
\usepackage{amsthm}
\usepackage{amssymb}
\usepackage{latexsym}
\usepackage[all]{xy}
\usepackage{pstricks,pst-node}
\usepackage[mathscr]{eucal}

\theoremstyle{plain}
\newtheorem{Thm}[equation]{Theorem}
\newtheorem{Cor}[equation]{Corollary}
\newtheorem{Prop}[equation]{Proposition}
\newtheorem{Lem}[equation]{Lemma}

\theoremstyle{definition}
\newtheorem{defn}[equation]{Definition}

\newtheorem{conj}[equation]{Conjecture}

\theoremstyle{remark}

\newtheorem{examp}[equation]{Example}
\newtheorem{rem}[equation]{Remark}
\newtheorem{claim}[equation]{Claim}

\makeatletter
\renewcommand{\subsection}{\@startsection{subsection}{2}{0pt}{-3ex
plus -1ex minus -0.2ex}{-2mm plus -0pt minus
-2pt}{\normalfont\bfseries}} \makeatother

\numberwithin{equation}{subsection}

\newcommand{\qgr}[1]{{\mathsf{Qgrmod}}(#1)}
\newcommand{\tail}[1]{{\mathsf{tails}}(#1)}
\newcommand{\Gr}[1]{{\mathsf{grmod}}(#1)}
\newcommand{\Lmod}[1]{#1\text{-}{\mathsf{mod}}}
\newcommand{\Rmod}[1]{\mathsf{mod}\text{-}#1}
\newcommand{\bimod}[1]{(#1)\text{-}{\sf{bimod}}}

\newcommand{\ree}[1]{{\operatorname{\textit{Rees}}}_{_{#1}\,}}

\newcommand{\hdot}{{\:\raisebox{2pt}{\text{\circle*{1.5}}}}}
\newcommand{\idot}{{\:\raisebox{2pt}{\text{\circle*{1.5}}}}}

\DeclareMathOperator{\Proj}{\mathrm{Proj}}
\DeclareMathOperator{\Wh}{\mathrm{Wh}}
\DeclareMathOperator{\whb}{{{\mathrm{Wh}}^\m_\m}}
\DeclareMathOperator{\Tor}{\mathrm{Tor}}
\DeclareMathOperator{\Ext}{\mathrm{Ext}}
\DeclareMathOperator{\sym}{\mathrm{Sym}}
\DeclareMathOperator{\var}{\textit{Var}}

\DeclareMathOperator{\im}{\mathrm{Im}}
\DeclareMathOperator{\supp}{\mathrm{Supp}}
\DeclareMathOperator{\coh}{\mathrm{Coh}}
\DeclareMathOperator{\Ker}{\mathrm{Ker}}

\DeclareMathOperator{\End}{\mathrm{End}}

\DeclareMathOperator{\rk}{\mathrm{rk}}
\DeclareMathOperator{\gr}{\mathrm{gr}}
\DeclareMathOperator{\modu}{\mathsf{mod}}

\DeclareMathOperator{\Ad}{\mathrm{Ad}}
\DeclareMathOperator{\ad}{\mathrm{ad}}
\DeclareMathOperator{\act}{\mathrm{act}}

\DeclareMathOperator{\ann}{\mathtt{Ann}}

\DeclareMathOperator{\qqgr}{\mathrm{qgr}}
\DeclareMathOperator{\invar}{\mathrm{Inv}}
\DeclareMathOperator{\loc}{{\scr{L}\!\textit{oc}}}

\newcommand{\dis}{\displaystyle}

\newcommand{\beq}{\begin{equation}\label}
\newcommand{\eeq}{\end{equation}}
\DeclareMathOperator{\Specm}{\mathrm{Specm}}
\DeclareMathOperator{\Spec}{\mathrm{Spec}}
\DeclareMathOperator{\pr}{pr}

\newcommand{\iso}{{\;\stackrel{_\sim}{\to}\;}}
\newcommand{\cd}{\!\cdot\!}
\DeclareMathOperator{\Hom}{\mathrm{Hom}}

\def\ccirc{{{}_{\,{}^{^\circ}}}}
\newcommand{\bo}{\mbox{$\bigotimes$}}
\renewcommand{\o}{\otimes }
\newcommand{\bplus}{\mbox{$\bigoplus$}}
\newcommand{\ccong}{\ \cong \  }
\newcommand{\ps}{\sigma }

\newcommand{\seq}{{\tiny\succcurlyeq}}

\newcommand{\wt}{\widetilde }

\newcommand{\Id}{{\operatorname{Id}}}

\def\npb{\noindent$\bullet\quad$\parbox[t]{149mm}}

\newcommand{\scr}[1]{\mathscr{#1}}
\newcommand{\vs}{\vskip 1pt }
\def\ccirc{{{}_{^{\,^\circ}}}}
\newcommand{\mmid}{\enspace\big|\enspace}
\newcommand{\la}{\lambda}
\newcommand{\be}{\beta }
\newcommand{\si}{\Sigma}
\newcommand{\s}{{\mathbb{S}}}
\newcommand{\sn}{{\mathcal{S}}}
\renewcommand{\ss}{{\widetilde{\mathcal{S}}}}

\newcommand{\zg}{{\mathfrak{Z}\g}}

\newcommand{\dd}{{\mathscr{D}}}
\newcommand{\oo}{{\mathcal{O}}}
\newcommand{\wh}{\widehat }
\renewcommand{\aa}{{\mathscr{A}}}
\newcommand{\cat}{{{(\dd_\nu,\mc)\mbox{-}\operatorname{mod}}}}
\newcommand{\catu}{{(\U_c,\mc)\mbox{-}\operatorname{mod}}}

\newcommand{\ug}{{\mathcal{U}}{\mathfrak{g}}}
\newcommand{\U}{{\mathcal{U}}}
\newcommand{\qq}{{\mathscr{Q}}}

\def\hp{\hphantom{x}}
\newcommand{\pa}{\partial }
\renewcommand{\H}{{\mathfrak h}}
\newcommand{\ZZ}{{\mathcal Z}}
\newcommand{\BB}{\B\times\B}
\newcommand{\tgr}{{\widetilde\gr}}
\renewcommand{\AA}{{A}}
\newcommand{\ba}{{\mathfrak A}}

\newcommand{\UU}{{\mathfrak U}}

\newcommand{\mc}{\m_\chi}
\newcommand{\al}{\alpha }

\newcommand{\WHC}{{\scr{W\!HC}}}
\newcommand{\hc}{{\scr H\scr C}}
\newcommand{\whc}{\textsl{wHC}}

\newcommand{\I}{{\mathfrak I}}
\newcommand{\IJ}{{\mathcal I}}
\newcommand{\C}{\mathbb{C}}
\newcommand{\g}{\mathfrak{g}}
\renewcommand{\b}{\mathfrak{b}}
\newcommand{\n}{\mathfrak{n}}
\newcommand{\m}{\mathfrak{m}}
\newcommand{\h}{\mathfrak{h}}
\newcommand{\La}{\Lambda }

\newcommand{\kk}{{\mathsf{K}}}
\newcommand{\op}{{\operatorname{op}}}
\newcommand{\bb}{{\scr B}}

\newcommand{\xx}{{\mathbb X}^+}

\newcommand{\mm}{{\mathcal M}}

\newcommand{\T}{{\mathsf T}}
\newcommand{\inv}{^{-1}}

\newcommand{\Z}{{\mathbb Z}}
\newcommand{\en}{{\enspace}}

\newcommand{\vi}{${\en\sf {(i)}}\;$}
\newcommand{\vii}{${\;\sf {(ii)}}\;$}
\newcommand{\viii}{${\sf {(iii)}}\;$}
\newcommand{\iv}{${\sf {(iv)}}\;$}

\newcommand{\sset}{\subset}
\newcommand{\sminus}{\smallsetminus}
\newcommand{\into}{\,\,\hookrightarrow\,\,}
\newcommand{\too}{\,\,\longrightarrow\,\,}
\newcommand{\mto}{\mapsto}
\newcommand{\onto}{\,\,\twoheadrightarrow\,\,}
\newcommand{\N}{{\mathcal{N}}}

\renewcommand{\O}{{\mathbb{O}}}
\newcommand{\Ga}{\Gamma }
\newcommand{\gm}{{\C^\times}}
\newcommand{\half}{\mbox{$\frac{1}{2}$}}
\newcommand{\om}{\omega }
\newcommand{\cv}{{\mathcal V}}
\newcommand{\B}{{\mathcal{B}}}
\newcommand{\LL}{{\scr L}}

\begin{document}
\title{Harish-Chandra bimodules for quantized Slodowy slices}
\author{Victor Ginzburg}\address{Department of Mathematics, University of Chicago,  Chicago, IL 
60637, USA.}
\email{ginzburg@math.uchicago.edu}

\dedicatory{To the memory of Peter Slodowy}

 \begin{abstract} 
The Slodowy slice is an especially nice slice to
a given nilpotent conjugacy class in a semisimple Lie algebra.
Premet introduced noncommutative quantizaions of the Poisson algebra of
polynomial
functions on the Slodowy slice. 

In this paper, we define and 
study  Harish-Chandra bimodules over
Premet's algebras. We apply the technique of
 Harish-Chandra bimodules to prove  a conjecture of Premet concerning
primitive ideals, to define  projective functors, and to construct
`noncommutative resolutions'
 of Slodowy slices via translation functors.
\end{abstract}

\maketitle

\centerline{\sf Table of Contents}

$\hspace{20mm}$ {\footnotesize \parbox[t]{115mm}{
\hp${}_{}$\!\hp1.{ $\;\,$} {\tt Geometry of Slodowy slices}\newline
\hp2.{ $\;\,$} {\tt Springer resulution of Slodowy slices}\newline
\hp3.{ $\;\,$} {\tt Quantization}\newline
\hp4.{ $\;\,$} {\tt Weak Harish-Chandra bimodules}\newline
\hp5.{ $\;\,$} {\tt $\scr D$-modules}\newline
\hp6.{ $\;\,$} {\tt Noncommutative resolutions  of Slodowy slices}
}}

\bigskip


\section{Geometry of Slodowy slices}

\subsection{Introduction}\label{e} Let $\g$ be a semisimple Lie
algebra, and $\ug$ the  universal enveloping algebra of $\g$.
For any nilpotent element $e\in\g$, Slodowy 
used the
Jacobson-Morozov theorem to construct a slice to the
conjugacy class of $e$ inside the nilpotent  variety of $\g$. This
slice,
$\sn$, has a natural  structure
of an affine algebraic Poisson variety.

More recently, 
 Premet \cite{premet} has
defined, following earlier works by
Kostant \cite{kostant},
Kawanaka \cite{Ka}, and Moeglin \cite{moeglin}, for each character ${c}$ of the center
of $\ug$,
 a filtered associative algebra $\AA_{c}$. The family of algebras
 $\AA_{c}$ may be thought of as a family of  quantizations of the Slodowy slice
$\sn$ in the sense that, for any ${c}$,
one has a natural Poisson algebra isomorphism $\gr \AA_{c}\cong \C[\sn]$.
For further developments see also \cite{BK1}, \cite{BGK}, \cite{L}, \cite{premet2}-\cite{premet3}.

The algebras $A_c$ have quite interesting
representation theory which is similar, in a sense, to the representation theory
of the Lie algebra $\g$ itself. It is well known that category
$\oo$ of Bernstein-Gelfand-Gelfand plays a key role in the
 representation theory
of $\g$. Unfortunately, there seems to be no reasonable analogue
of category $\oo$ for the algebra $A_c$, apart from some special
cases, cf. \cite{BK1}, \cite{BGK}. 

In this paper, we propose to remedy the above mentioned difficulty by introducing
a category of (weak) Harish-Chandra {\em bi}modules over the algebra $A_c$.
Our definition of  weak Harish-Chandra  bimodules actually makes sense
for a wide class of algebras, cf. Definition \ref{whc_def}.
We show in particular that,
in the case of  enveloping algebras, a weak Harish-Chandra
$\ug$-bimodule is a Harish-Chandra  bimodule in the conventional sense
(used in representation theory 
of semisimple Lie algebras for a long time) if and only if
the corresponding $\dd$-module on the flag variety has regular singularities,
cf. Proposition \ref{RS}. Motivated by this result,
we  use  `micro-local' technique
developed
in \S\ref{nc_sec} to introduce a notion of `regular singularities'
for $A_c$-bimodules.\footnote{Our definition of regular singularities
works more generally, for modules over an arbitrary  filtered
$\C$-algebra  $A$  having the property that $\gr A$ is a finitely generated commutative
algebra such that the Poisson scheme
$\Spec(\gr A)$ admits a symplectic resolution.} We then define  Harish-Chandra 
$A_c$-bimodules as weak Harish-Chandra bimodules
with regular singularities.

Let $\I_c$ be the maximal ideal of the center of
the algebra  $\ug$ corresponding to a 
{\em regular} central character
$c$.
There is an associated  block $\oo_c$, of the category 
 $\oo$ of Bernstein-Gelfand-Gelfand, formed by the
  objects $M\in \oo$ such that the $\I_c$-action on $M$ is
nilpotent. One may also consider an analogous
category
of  Harish-Chandra $\ug$-bimodules. Specifically, one
considers  Harish-Chandra
 $\ug$-bimodules $K$ such that
the left $\I_c$-action on $K$ is nilpotent and  such that
one has $K\I_c=0$.
 It is known that the resulting category
 is, in fact, {\em equivalent}
to the category $\oo_c$, see \cite{BG}. Thus, one may expect the
category of  Harish-Chandra $A_c$-bimodules to be the
right substitute for a  category $\oo(A_c)$ that
may or may not exist.

For any algebra $A$, a basic example of a weak   Harish-Chandra $A$-bimodule
is the algebra $A$ itself, viewed as the diagonal bimodule.
Sub-bimodules of  the diagonal bimodule are nothing but
two-sided ideals of $A$. This shows that the theory
of (weak)   Harish-Chandra $A$-bimodules is well suited for studying
ideals in $A$, primitive ideals, in particular.

We construct an analogue of the Whittaker functor
from the category of   Harish-Chandra $\ug$-bimodules
to the category of
  Harish-Chandra $A_c$-bimodules.
Among our most important results are Theorem \ref{main_thm}
and Theorem \ref{adj} which
describe key properties of that functor.

We use the above results  to provide an alternative proof of
a conjecture of Premet that  relates  finite dimensional $A_c$-modules
to primitive ideals  $I\sset\ug$ such that
the associated variety of $I$
equals $\overline{\Ad G(e)}$, the closure of the conjugacy class of the
nilpotent $e\in\g$. This conjecture
was proved in a special case  by Premet \cite{premet3},
 using reduction to positive characteristic,
and  in full generality by Losev \cite{L}, using
deformation quantization, and later also by Premet \cite{premet4}.
Our approach  
is totally different from the aproaches used by Losev or Premet
and is, in a way, more straightforward.
Some results closely related to ours  were also obtained
by Losev \cite{L2}.



Finally, we introduce translation functors
on representations of the algebra $A_c$. In \S\ref{nc_sec},
we use those functors to   construct
`noncommutative resolutions' of the Slodowy slice by means of
a noncommutative  Proj-construction. Similar construction has
 been successfully exploited
earlier, in other
situations, by Gordon and Stafford \cite{GS}, and by Boyarchenko
\cite{Bo}.

 A different approach 
to  `noncommutative resolutions' of  Slodowy slices
was also proposed by Losev in an unpublished manuscript.

\begin{rem} We expect that, in the special case
of {\em subregular} nilpotent elements,
 our constructuion of  noncommutative resolution reduces to that of 
Boyarchenko. In more detail,
Boyarchenko considered noncommutative resolutions of
certain noncommutative algebras introduced
by Crawley-Boevey and Holland \cite{CBH}. These algebras are quantizations of
the coordinate ring of a Kleinian singularity.
By a well known result of Brieskorn and Slodowy \cite{slodowy},
the Slodowy slice to the subregular  nilpotent in
a simply laced Lie algebra $\g$
is isomorphic, as an algebraic variety, to
the  Kleinian singularity associated with the Dynkin diagram of $\g$.
Furthermore, it is expected (although no written proof
of this seems to be available, except for type $\mathbf A$,
see \cite{Hod})
 that the  algebras $A_c$
are, in that case, isomorphic to the algebras constructed
by Crawley-Boevey and Holland. Thus, our  noncommutative resolutions
should correspond, via the isomorphism, to those considered
by Boyarchenko.
\end{rem}
\subsection*{Acknowledgments.}
{\footnotesize{I am very much indebted to Ivan Losev for initiating this work,
 for very useful discussions, and for explaning to me the  results
of \cite{L2} before that work was made public.
I also thank Mitya Boyarchenko for
explaining
to me some details of his work \cite{Bo}, and Iain Gordon for a
careful reading of a preliminary draft of this paper.
This work was  supported in part by the NSF   grant  DMS-0601050.}}

\subsection{}\label{1}
Let $\g$ be a complex semisimple Lie algebra
and  $G$ the adjoint group
of $\g$. 

From now on, we fix an
$\mathfrak{sl}_{2}$-triple $\{e,h,f\}\sset\g,$ equivalently, a 
Lie algebra imbedding  $\mathfrak{sl}_{2}\into\g$
such that ${\footnotesize\left(\begin{matrix}
0&1\\0&0\end{matrix}\right)}\mto e$ and
${\footnotesize\left(\begin{matrix}
1&0\\0&-1\end{matrix}\right)}\mto h$.
Slodowy (as well as Harish-Chandra) showed, see e.g. \cite[\S 7.4]{slodowy},
 that the affine linear space $e+\Ker(\ad f)$
is a transverse slice to the $\Ad G$-conjugacy class of $e$, a
nilpotent element
of $\g$.

The Killing form on $\g$ provides an $\Ad G$-equivariant isomorphism 
$\kappa: \g\iso \g^*$, where
$\g^*$ is  the vector space dual to $\g$. Put
 $\chi:=\kappa(e)\in\g^*$.
Write $\N\sset \g^*$
for the image of the set of nilpotent elements of
$\g$, resp.   $\s\sset\g^*$ for the image of the set $e+\Ker\ad f$,
 under the isomorphism $\kappa$. Thus, $\s$ is   a transverse slice at  $\chi$, to be
called the {\em Slodowy slice}, to
the coadjoint orbit $\O:=\Ad G(\chi)$.

Below, we   mostly restrict our attention to the scheme theoretic intersection
 $\sn:=\s\cap\N$. 
Extending some classic results of Kostant, Premet proved
the following,  \cite[Theorem 5.1]{premet}.

\begin{Prop}\label{kost} The scheme $\sn$ is reduced, irreducible, and
Gorenstein.
Moreover, it is a
 normal complete intersection in $\s$ of dimension $\dim\N-\dim\O$.\qed
\end{Prop}

\subsection{The Lie algebra $\m$}\label{2} 
The  $\ad h$-action on $\g$  yields a weight decomposition
\beq{weight}
\g = \bplus_{i
\in \mathbb{Z}}\ \g(i)\quad\text{where}\quad\g(i) = \{ x \in \g \mmid [h,x]=ix \}.
\eeq

Note that  $e$ is a nonzero element of $\g(2)$.
Hence, the assignment $\g(-1)\times \g(-1)\to\C,\
x,y\mto \chi([x,y]),$ gives
a skew-symmetric  nondegenerate bilinear form  on $\g(-1)$.
We choose and fix   $\ell\sset\g(-1)$, a Lagrangian subspace  with respect to that form.

Following Kawanaka \cite{Ka} and Moeglin \cite{moeglin}, one puts 
\begin{displaymath}
\m := \ell\ \bplus\ \bigl(\oplus_{_{i \leq -2}}\; \g(i)\bigr).
\end{displaymath}

Thus, $\m $ is a nilpotent Lie subalgebra of
$\g$ such that  the linear function $\chi$ 
vanishes on $[\m,\m]$. Let $M$ be the unipotent subgroup of $G$ with Lie
algebra $\m$. By a standard easy computation one finds, cf. \eqref{dimm},
\beq{dimM}
\dim M=\half\dim\O.
\eeq

 Write $V^\perp\sset \g^*$ for
the annihilator of a vector subspace $V\sset \g.$
We will prove

\begin{Prop}\label{ggnil} \vi Any $\Ad G$-orbit in $\g^*$ meets the affine 
space $\chi+
\m^{\perp}$
transversely. \vs

\vii  We have $\O\cap(\chi+ \m^{\perp})=\Ad M(\chi)$
is a closed
 Lagrangian submanifold in the coadjoint orbit $\O=\Ad G(\chi)$,  a symplectic manifold
with  Kirillov's symplectic
structure. \vs

\viii The scheme theoretic intersection $\,\N\cap (\chi+ \m^{\perp})$ is a
 reduced complete intersection in $\chi+ \m^{\perp}$;
 the  $M$-action induces an $M$-equivariant isomorphism 
$\ \dis M \times \sn\ \iso\ \N\cap (\chi+ \m^{\perp}).$
\end{Prop}

\begin{rem}
Using  the isomorphism in the last line above, we may write
\beq{red_var}
\sn\ccong [\N\cap (\chi+\m^\perp)]/\!\Ad M.
\eeq

This formula means that the variety  $\sn$ may be obtained from $\N$
by a Hamiltonian reduction with respect to the natural  $M$-action on $\N$.
\end{rem}

\begin{conj}\label{LL} There exists a  Borel subalgebra $\b\sset\g$
such that $\b^\perp\cap \Ad M(\chi)=\{\chi\}$.
\end{conj}

\begin{rem} We recall that, for any Borel subalgebra $\b$,
(each irreducible component of) the set $(\b^\perp\cap\O)_{\text{red}}$
 is known to be a Lagrangian subvariety
in $\O$, see \cite[Theorem 3.3.6]{CG}.

Thus,  Conjecture \ref{LL} says that there exists a  Borel subalgebra $\b$ such
that  the two  Lagrangian subvarieties $\b^\perp\cap\O$ and $\Ad M(\chi)$
meet at a single point $\chi$.\end{rem}
\smallskip

Let $\varpi:
\g^*\to\m^*$ be the canonical projection induced by restriction of
linear functions from $\g$ to 
$\m$. The map $\varpi$ may be thought of as a moment map associated
with the $\Ad M$-action on $\g^*$.
We put  $\chi_\m:=\chi|_\m=\varpi(\chi).$
Observe that $\chi_\m\in\m^*$ is  a fixed point of
 the coadjoint $M$-action on $\m^*$, and we have
$\varpi\inv(\chi_\m)=\chi+\m^\perp.$

We write $\oo_Y$ for the structure sheaf of a scheme $Y$.
Let $\oo_{\chi}$ denote the localization
of the polynomial algebra $\C[\m^*]$ at the point $\chi_\m$.

\begin{Cor}\label{flat}  Let $Y$ be a $G$-scheme, and let $f: Y\to\g^*$
be a  $G$-equivariant morphism such that $\chi\in f(Y).$ Then, we have

\vi  The point $\chi_\m$ is a regular value of the
composite map $\varpi\ccirc f$.

\vii  Given a $G$-{\em equivariant} coherent $\oo_Y$-module $M$,
the localization 
$(\varpi\ccirc f)^*\oo_\chi\bigotimes_{\oo_Y}M$
 is a flat  $(\varpi\ccirc f)^\hdot\oo_\chi$-module.

\viii If $Y$ is reduced and irreducible, then 
$f\inv(\s)$ is a reduced complete intersection in $Y,$ of dimension
$\dim f\inv(\s)=\dim Y-\dim\O$.\qed
\end{Cor}
\begin{proof}
The transversality statement of Proposition \ref{ggnil}(i) may be
equivalently reformulated as  follows:

{\em For any $x\in \g^*$,
the point $\chi_\m$ is a regular value of the
composite map} $\Ad G(x)\into\g^*\stackrel{\varpi}\to\m^*$.
\vskip 2pt

The above statement insures that, for any point
$y\in Y$ such that $\varpi(f(y))=\chi_\m$, the differential
$d\varpi\ccirc df:
T_yY\to T_{\chi_\m}\m^*,$
of the map $\varpi\ccirc f$, is a surjective
linear map. This yields part (i) of the corollary.
Parts (ii)-(iii) follow  from (i) combined with  Proposition \ref{ggnil}(i).
\end{proof}

\subsection{Proof of  Proposition \ref{ggnil}}
First, we recall a few
well known results. Write $\g^x\sset \g$ for the centralizer of
an element $x\in\g$. 
To prove  formula \eqref{dimM},
we compute
$$
\dim\m=\dim\big(\bplus_{i<0}\ \g_i\big) -\half\dim\g(-1)
=
\half[\dim\g-\dim\g(0)]-\half\dim\g(-1).
$$

By  ${\mathfrak{sl}}_2$-theory, one has that $\dim\g(0)+\dim\g(1)=\dim\g^e$.
Hence, we find
\beq{dimm}
\dim\m=\half[\dim\g-\dim\g(0)-\dim\g(1)]=
\half[\dim\g-\dim\g^e]=\half\dim\O.\en\Box
\eeq
\smallskip

Observe next that, for the element $f$ of our $\mathfrak{s}\mathfrak{l}_2$-triple,
 we have $\g^f\sset \bplus_{i\leq 0}\
\g(i),$ by  $\mathfrak{s}\mathfrak{l}_2$-theory. It follows that
$\s\sset\chi + \m^{\perp}$. Also, it is easy to see that the set
$\chi + \m^{\perp}$ is stable under the coadjoint $M$-action.
Moreover, it was proved in \cite[Lemma 2.1]{GG} 
that the $M$-action in $\chi + \m^{\perp}$ is free,
and the action-map induces an $M$-equivariant isomorphism of algebraic
varieties:
\beq{isom_S}
M \times \s\ \iso\ \chi+ \m^{\perp},
\eeq
where $M$ acts on $M \times \s$ via its action on the first factor by
left translations.
\smallskip

We may exponentiate the Lie algebra map $\mathfrak{sl}_{2}\to\g$
to a rational group homomorphism $SL_2\to G$.
Restricting the latter map
to the torus $\gm\sset SL_2$, of diagonal matrices, one gets a
morphism $\gamma: \gm\to G,\ t\mto\gamma_t$. Following Slodowy, one
defines a $\bullet$-{\em action}  of $\C^\times$ on $\g$ by
\beq{bullet}
\gm\ni t:\ x\mto t\bullet x:=t^{2}\cdot
\Ad \gamma_{t^{-1}}(x),\qquad
x \in \g.
\eeq

Since $\Ad \gamma_t (e) = t^{2}\cdot e$,
the $\bullet$-action
 fixes $e$. 
Dualizing,
one gets
a $\bullet$-action of $\gm$ on $\g^*$ that  fixes the point
$\chi\in\g^*$.
It is easy to see that each of the spaces, $\chi+\s$ and $\chi+\m^\perp$, is $\bullet$-stable
and, moreover, the $\bullet$-action contracts the space $\chi+\m^\perp$ to $\chi$.

\begin{proof}[Proof of Proposition \ref{ggnil}] Observe
first  the  statement of part  (i)
is clear for the orbit $\O=\Ad G(\chi),$ since the space
$\chi+\m^\perp$ contains a transverse slice to $\O$.
Below,  we will use the identification
$\kappa: \g\iso\g^*$, so $\chi$ gets identified with  $e$, and we may 
 write $\chi+\m^\perp\sset \g$.

Now, let $x\in \chi+\m^\perp$ be an arbitrary element and put
$O:=\Ad G(x)$.
We are going to reduce the statement  (i) for $O$ to the special
case of the orbit $\O$
using the $\bullet$-action as follows.

Observe that
the 
tangent space to $O$  at the point $x$
equals $T_xO=[\g,x]=\ad x(\g),$ a vector subspace of $\g$.
Proving part (i) amounts to showing that
the composite   $\pr\ccirc\kappa\ccirc\ad x:\ \g\to \g\iso\g^*\onto\m^*$ is
a surjective linear map, for any $x\in \chi+\m^\perp$.
It is clear that, for any  $x\in \g$ sufficiently close to
$e$, the map $\pr\ccirc\kappa\ccirc\ad x$ is surjective, by continuity.
Since the $\bullet$-action on $\chi+\m^\perp$
is a contraction,
we deduce by $\gm$-equivariance that the surjectivity holds 
for any $x\in e+\m^\perp$.
Part (i) is proved.

The isomorphism of part (iii) follows from Proposition
\ref{kost},
by  restricting   isomorphism
\eqref{isom_S} to $\N$. All the other claims of part (iii) then follow from
the  isomorphism. 

To prove (ii), we observe first, that since
$\chi$ vanishes on $[\m,\m]$, the vector space $\ad\m(\chi)$ is an
isotropic
subspace of the tangent space $T_\chi\O$. This implies, by
$M$-equivariance,
that $\Ad M(\chi)$ is an isotropic submanifold of $\O$. 
This $\Ad M$-orbit is closed since $M$ is a unipotent group.
Finally, the $M$-action being free
we find $\dim \Ad M(\chi)=\dim M=\half\dim\O$, by formula \eqref{dimM}.
It follows that
$\Ad M(\chi)$ is a Lagrangian submanifold of $\O$.
\end{proof}

\section{Springer resulution of Slodowy slices}
\subsection{The Slodowy variety}\label{SV} Let $\B$ be the flag variety, i.e.,
the variety of all Borel subalgebras in $\g$. Let $T^*\B$ be the total
space of the
 cotangent bundle on $\B$, equipped with the standard symplectic
 structure
and the natural Hamiltonian $G$-action.
An associated moment map  is  given by the first projection
\beq{TB}
\pi:\ T^*\B=\{(\la,\b)\in\g^*\times\B\mmid \la\in\b^\perp\}\too\g^*,
\quad (\la,\b)\mto\la.
\eeq

The map $\pi$, called   Springer resolution,
is a {\em symplectic resolution} of
$\N.$ This means that, one has
 $\N=\pi(T^*\B)$ and, moreover,
 $\pi$ is a resolution of singularities of $\N$
such that the pull-back
morphism $\pi^*:\ \oo_{\g^*}\to\oo_{T^*\B}$ intertwines
the Kirillov-Kostant Poisson bracket on $\oo_{\g^*}$ with the
 Poisson bracket on $\oo_{T^*\B}$ coming from the
symplectic structure on $T^*\B$.

The {\em Slodowy variety} is defined as
  $\ss:=\pi\inv(\sn)=\pi\inv(\s)$, a scheme theoretic preimage of the Slodowy slice
under the Springer resolution.
Also, let $\B_\chi:=\pi\inv(\chi)_\text{red}$ be the Springer   fiber over $\chi$,
equipped with {\em reduced} scheme structure. Clearly, we have
$\B_\chi\sset\ss$.

\begin{Prop}\label{red} \vi The map $\pi: \ss\to \sn$
is a symplectic resolution,  in particular, $\ss$ is
a smooth  and connected symplectic  submanifold in $T^*\B$
of dimension $\dim\ss=2\dim\B_\chi.$

\vii The  Springer fiber $\B_\chi$ is a  (not necessarily irreducible)
Lagrangian subvariety of $\ss$.
\end{Prop}

Our next goal is to find a  Hamiltonian reduction construction
for the variety $\ss$. 
Specifically, we would like to get
an analogue of formula \eqref{red_var}
where  the variety $\sn$ is replaced by $\ss$ and where
the symplectic manifold $T^*\B$ plays the role of the Poisson variety $\N$.
To do so, it is natural to try to  replace, in  formula \eqref{red_var},
the space $\chi+\m^\perp$ by $\pi\inv(\chi+\m^\perp)$.
Thus, we are led to introduce a scheme
\beq{sigma}
\si:=\pi\inv(\chi+ \m^{\perp})=
(\varpi\ccirc\pi)\inv(\chi_\m) \sset T^*\B.
\eeq

\begin{Prop}\label{red_si}
 \vi The scheme $\si$ is a (reduced) smooth  connected
manifold, and we have
$\dim\si=\dim\B+\dim\B_\chi=\dim T^*\B-\dim\m.$

\vii   The scheme $\si$ is  $M$-stable,
and  the action-map induces an  $M$-equivariant isomorphism
\beq{isom_si}
M \times \ss\ \iso\ \si.
\eeq

\viii  The scheme $\si$ is a coisotropic submanifold in $T^*\B$;
the  nil-foliation on $\si$ coincides with the fibration by
$M$-orbits.
\end{Prop}

\subsection{Proof of Propositions \ref{red} and \ref{red_si}}
There is a natural $\gm$-action on $T^*\B$ along the fibers 
of the cotangent bundle projection $T^*\B\to\B$. 
 Thus, the group $G\times \gm$ acts on $T^*\B$ by
$(g,z):\ x\mto z\cdot g(x)$.
Further, we define a $\bullet$-action of the torus $\gm$  on $T^*\B$
by the formula $\gm\ni t:\ x\mto t\bullet x:=t^{2}\cdot
\gamma_{t^{-1}}(x),$
cf. \eqref{bullet}.

The Springer resolution \eqref{TB}
is a  $G\times \gm$-equivariant morphism.
Hence, the map $\pi$ commutes with the $\bullet$-action
as well. 
It follows in particular that  $\B_\chi,\ \ss,$ and $\si$,
are all
$\bullet$-stable subschemes of $T^*\B$. 
The $\bullet$-action retracts $\chi+\m^\perp$ to $\chi$, hence,
provides a retraction of $\si=\pi\inv(\chi+\m^\perp)$ to $\B_\chi=\pi\inv(\chi)$.

Further, a result of Spaltenstein \cite{Spa} says
 that  the Springer fiber is a connected variety
and, moreover, all  irreducible
components of $\B_\chi$  have the same dimension, 
 cf. also \cite[Corollaries 7.6.16 and 3.3.24]{CG}, which is equal ~to
\beq{dim_Bchi}
\dim\B_\chi=\dim\B-\half\cdot\dim\O.
\eeq

\begin{proof}[Proof of Proposition \ref{red_si}]
Corollary \ref{flat}(i) insures that the point $\chi_\m\in\m^*$ is a regular
value of
the   map $\varpi\ccirc\pi: T^*\B\to \m^*$.
It follows, in particular, that $\si$ is a (reduced) smooth subscheme
of $T^*\B$ and that $\dim\si=\dim T^*\B-\dim\m.$
Furthermore, \eqref{isom_si} holds as a scheme theoretic isomorphism and
we have $\dim\si=\dim\ss+\dim\m$.

Since $\B_\chi$ is  connected and the
$\bullet$-action contracts $\ss$, resp. $\si$, to $\B_\chi$, we deduce that
$\ss$, resp.
$\si$, is a
connected manifold. This completes the proof of Proposition \ref{red_si}(i).

Part (ii) of the proposition 
is immediate from the isomorphism of Corollary \ref{ggnil}(i). Part
(iii) is
a general property of the fiber  of a moment map over a regular value, 
see eg. \cite{GuS}.
\end{proof}

\begin{proof}[Proof of Proposition \ref{red}] First of all,
we observe that
the smoothness of $\si$, combined with
\eqref{isom_si}, implies that $\ss$ is a smooth scheme. Furthermore,
from part (iii) of  Proposition \ref{red_si} and the isomorphism
$\ss\cong\si/M$ we deduce that the symplectic 2-form on $T^*\B$ restricts
to a nondegenerate 2-form on $\ss$. This yields Proposition \ref{red}(i).

Further, we know that  $\pi: \ss\onto\sn$ is a projective and dominant morphism
which is an isomorphism  over the open dense subset of
$\sn$ formed by regular nilpotent elements. It follows
that the map  $\pi: \ss\onto\sn$ is a symplectic resolution.
The scheme $\sn$ being irreducible,
we deduce that $\ss$ is connected (we have already proved this fact
differently
in the course of the proof of Proposition \ref{red_si}).

By Proposition
\ref{kost}, we get
$\dim\ss=\dim\sn=\dim\g-\rk\g-\dim\O.$ Hence,
using \eqref{dim_Bchi} and the equality $2\dim\B+\rk\g=\dim\g$,
we find
$\dim\ss=(2\dim\B+\rk\g)-\rk\g-2(\dim\B-\dim\B_\chi)=2\dim\B_\chi$.
This completes the proof of part (i) of Proposition \ref{red}.

To prove part (ii), let $y=(\b,\chi)\in\B_\chi$. Recall that any tangent vector
to $T^*\B$ at $y$ can be written in the form
$\ad^* a(y)+ \alpha$, for some $a\in\g$  and  some vertical
vector $ \alpha\in\b^\perp$ (i.e. a vector tangent to the fiber of the cotangent bundle).
It is clear that such a vector $\ad^* a(y)+ \alpha$ is tangent to
$\B_\chi$
if and only if one has $\alpha+\ad^* a(\chi)=0.$
Now, let $\ad^* b(y)+ \be$ be a second tangent vector at $y$  which is tangent to
$\B_\chi$
at the point $y$. Thus, we have $\ad^* a(\chi)=- \alpha$ and
$\ad^* b(\chi)=- \be$.

Using an explicit formula  for the symplectic 2-form $\om_{_{T^*\B}}$
on $T^*\B$, see eg. \cite[Proposition 1.4.11]{CG},   we find
\begin{multline*}
\om_{_{T^*\B}}(\ad^* a(y)+ \alpha,\ \ad^* b(y)+ \be)=
\chi([a,b])+ \alpha(b)-\be(a)\\
=\chi([a,b])-(\ad^* a(\chi))(b)+
(\ad^* b(\chi))(a)\\
=\chi([a,b])-\chi([a,b])-\chi([a,b])=
-\chi([a,b]).
\end{multline*}

Next, let  $\om_\O$ denote  the Kirillov-Kostant 2-form on the orbit
$\O$.
By definition, we have $\om_\O(\ad^* a(\chi), \ \ad^* b(\chi))=\chi([a,b])$.
Further, it is known that $\O\cap \b^\perp$ is an isotropic
subvariety in $\O$, cf. \cite[Theorem 3.3.7]{CG}.
The vectors $\ad^* a(\chi)=- \alpha,\
\ad^* b(\chi)=- \be\in\b^\perp$ are clearly tangent to that subvariety.
We deduce
$0=\om_\O(\ad^* a(\chi), \ \ad^* b(\chi))=\chi([a,b]).$ 
\end{proof}

\begin{rem}\label{slice} The smoothness statement in Proposition \ref{red}(i) is a special
case of the following elementary general result.

Let $X$ be a $G$-scheme, and let $F: X\to\g^*$ be a
$G$-equivariant morphism
such that $\chi\in F(X)$. Set ${\wt\s}:=F\inv(\s)$,
a scheme theoretic preimage of  $\s$,
and write $F_\s=F|_{{\wt\s}}:\ {\wt\s}\to \s$.

Then, for  any $x\in F\inv(\chi)$,
there are local isomorphisms
$(X, x)\iso (\O,\chi)\times ({\wt\s}, x)$, resp.
$(\g^*,\chi)\iso(\O,\chi)\times (\s,\chi)$,  in {\em \'etale} topology,
such that the map $F$ goes, under the isomorphisms, to
the map $\Id_\O\times F_\s: \ \O\times {\wt\s}
\to \O\times\s$.
\end{rem}

\subsection{}\label{tw_sec} Given a manifold $Y$, we write
$p:\ T^*Y\to Y$ for  the cotangent bundle projection.
 
Let $X\sset Y$
be a submanifold. Below, we will use the following
\begin{defn}  A submanifold $\Lambda\sset p\inv(X)$ is said
to be  a {\textsf{twisted conormal bundle}}
on $X$ if $\Lambda$ is a Lagrangian submanifold of $T^*Y$ and, moreover,
the map $p$ makes the projection
$\Lambda\to X$ an affine bundle.
\end{defn}

Let  $\ps$ be the restriction to $\si\sset T^*\B$
of the projection
 $p: T^*\B\to\B$.
It is clear that
$\ps(\si)$ is an $M$-stable  subset of $\B$, and one has the
following diagram of $M$-equivariant maps
\beq{diagram}
\xymatrix{
&\si\ar@{->>}[dl]_<>(.5){\ps}\ar@{->>}[dr]^<>(.5){\pi}&\\
\ps(\si) &&\N\cap(\chi+ \m^{\perp})
}
\eeq

\begin{Prop}\label{twisted} \vi  For a Borel subalgebra $\b$,
we have: $\, \b\in \ps(\si)\en\Leftrightarrow\en
\chi|_{\m\cap\b}=0.$
\vskip 1pt

\vii For any $M$-orbit $X\sset \ps(\si)$, the map $\ps$
makes the projection $\ps\inv(X)\to X$  a twisted conormal bundle
on the submanifold $X\sset\B$.
\end{Prop}
\begin{rem} 
Let $\B^\circ$ be the set of all Borel subalgebras $\b$ such that $\b\cap\m=0$.
It is easy to see that  $\B^\circ$ is an $M$-stable,  Zariski open and
dense subset of $\B$. Part (i) of Proposition \ref{twisted} implies
that we have  $\B^\circ\sset \ps(\si).$
\end{rem}

\begin{proof} Clearly,  $\b\in \ps(\si)$ if and only if the fiber
$\ps\inv(\b)$ is non-empty.
By definition, we have
\beq{p_fib}
\ps\inv(\b)=\{(\la,\b)\in\g^*\times\B\mmid \la\in\b^\perp
\enspace\text{and}\enspace \la\in \chi+\m^\perp\}\cong \b^\perp
\cap(\chi+\m^\perp).
\eeq
Note that  $\chi|_{\m\cap\b}=0$ says that  $\chi\in (\b\cap\m)^\perp.$
Thus, the equivalences below yield  part (i),
$$\b\in \ps(\si)
\enspace \Longleftrightarrow\enspace
\ps\inv(\b)\neq\emptyset
\enspace \Longleftrightarrow\enspace
\b^\perp
\cap(\chi+\m^\perp)\neq\emptyset
\enspace \Longleftrightarrow\enspace
\chi\in \b^\perp+\m^\perp=(\b\cap\m)^\perp.
$$

To prove (ii), fix a point $\b\in \ps(\si)$,
and let $X:=M\cdot\b\sset\B$ be the $M$-orbit of the point $\b$.
Thus, writing $B$ for the Borel subgroup  corresponding
to $\b$, we have
$X\cong M/M\cap B$.
Further, we may use the last displayed formula above and choose
$\la_\b\in\b^\perp$ and $\mu_\b\in\m^\perp$ such that
$\chi=\la_\b+\mu_\b$.

\begin{claim}\label{cl}
The map $\ps: \ps\inv(X)\to X$
is $M$-equivariantly isomorphic to the fibration
$\phi:\ M\times_{M\cap B}\bigl(\la_\b+(\m^\perp\cap\b^\perp)\bigr)\to
M/M\cap B$.
\end{claim}
 The above claim implies part (ii) since
the   fibration $\phi$, of the claim, is well known to
be a twisted conormal bundle
on the submanifold $X\sset\B=G/B$, cf. \cite[Proposition 1.4.14]{CG}.

To prove Claim \ref{cl}, we write
 $\chi+\m^\perp=\la_\b+\mu_\b+\m^\perp=\la_\b+\m^\perp.$
Hence, using \eqref{p_fib},
we may write
$$\ps\inv(\b)=\b^\perp
\cap(\chi+\m^\perp)=\b^\perp
\cap(\la_\b+\m^\perp)=\{\la\in\b^\perp\mmid \la-\la_\b\in\m^\perp\}=
\la_\b+(\b^\perp\cap\m^\perp),
$$
which is an affine-linear subspace in $\b^\perp$
obtained from the vector space $\b^\perp\cap\m^\perp\sset \b^\perp$
by translation. It follows that
$\ps\inv(M\cdot \b)\cong M\times_{M\cap B}\ps\inv(\b)\cong
M\times_{M\cap B}\bigl(\la_\b+(\m^\perp\cap\b^\perp)\bigr)$.

Now, we  identify
the tangent space to $\B$, resp. to the $M$-orbit $X$, at the point $\b$ with 
$T_\b\B=\g/\b$, resp. 
with $T_\b{X}=(\m+\b)/\b\sset\g/\b$.
We see that the vector space $\b^\perp\cap\m^\perp=(\m+\b)^\perp$
equals the annihilator of $T_\b{X}$ in $T_\b\B.$
This proves Claim \ref{cl}, and part (ii) follows.
\end{proof}

\section{Quantization}
\subsection{The Premet algebra}
Given a Lie algebra ${\mathfrak k}$, we write
 $\sym {\mathfrak k}$
for the Symmetric, resp.
  $\U{\mathfrak k}$ for the enveloping,
algebra of  ${\mathfrak k}$.
 We keep the notation of the previous section, fix $\chi\in \N$,
and define a linear map $\m\to\sym \g$,
 resp. $\m\to\ug$, by the assignment $m\mto m-\chi(m)$. Let $\m_\chi$
denote the image. Using the canonical
isomorphism $\C[\g^*]=\sym \g$, one can identify
$\mc$ with the space of degree $\leq 1$ polynomials
on $\g^*$ vanishing on $\chi+\m^\perp.$

Let $\zg$ denote the center of $\ug$
and write $\Specm\zg$ for the set of maximal ideals of the
algebra $\zg$. 
Given  $c\in\Specm\zg$, let $\I_c\sset \zg$ denote the corresponding
maximal ideal,
and put $\U_c:=\ug/\ug\cdot\I_c$.
Let $\U_c\cdot\mc$ denote the left ideal of the algebra $\U_c$
generated by the image of the composite
$\mc\into\ug\onto\U_c$.

Similarly, let $I_o$ denote  the augmentation ideal of the algebra
 $(\sym \g)^{\Ad G},$ and put $S_o:=\sym \g/\sym \g\cdot I_o$. Thus, $S_o=\C[\N]$.
Let $S_o\cdot\mc$ denote an ideal of the algebra $S_o$
generated by the image of the composite
$\mc\into\sym\g\onto S_o$. It is clear that we have 
\beq{put1}\C[\N\cap(\chi+\m^\perp)]=S_o/S_o\cd\mc;
\quad\text{by analogy, we put}\quad
Q_c:=\U_c/\U_c\cd\mc.
\eeq

The left ideal  $\U_c\cdot\mc\sset \U_c$, resp. the  ideal 
$S_o\cd\mc\sset S_o$, is 
$\ad\m$-stable.
It is clear that the scheme isomorphism \eqref{red_var} translates into the
following  algebra isomorphism
\beq{put2}
\C[\sn]\cong
(S_o/S_o\cd\mc)^{\ad \m};\quad\text{by analogy, we put}\quad
A_c:=(\U_c/\U_c\cd\mc)^{\ad \m}.
\eeq

It is easy to see that multiplication in $\ug$ gives rise
to a well defined (not necessarily commutative)
associative algebra structure on $A_c$. Furthermore,
the above  formulas show that the algebra $\C[\sn]$
is obtained from   $\C[\N]$ 
by a
{\em classical Hamiltonian reduction},
resp. the algebra $A_c$ is obtained from 
$\U_c$ by a
{\em quantum  Hamiltonian reduction}.
 In particular, $\C[\sn]$  has a natural  Poisson algebra
 structure.
The algebra $A$
The family of algebras $\
\{A_c,\ c\in \Specm\zg\}$  may be viewed
as `quantizations' of the Poisson algebra $\C[\sn].$

\begin{rem}\label{A} Each of the algebras $A_c$ is a quotient
of a single algebra $A:=(\ug/\ug\cdot\mc)^{\ad\m}$.
This algebra $A$, that has been introduced and studied
by  Premet
in \cite{premet}, is
a Hamiltonian reduction of the algebra
$\ug$. The natural imbedding $\zg\into\ug$ descends to
a well-defined algebra map $\jmath:\ \zg\to A$ with central image.
It is easy to see that,
for any central character $c$, one has $A/A\cdot\jmath(\I_{c})=A_{c}$.
\end{rem}

It is immediate to check that, for any right $\U_c$-module
$N$, the assignment $u:\ n\mto nu$ induces a well defined right
$A_c$-action on the coinvariant space $N/N\cd\mc.$
This yields, in particular, a right
$A_c$-action on $Q_c$, cf. \eqref{put1}, that commutes with the  natural left
$\U_c$-action. In this
way, $Q_c$ becomes an $(\U_c,A_c)$-bimodule. Moreover, the right
action of $A_c$
gives
an algebra
isomorphism  $A_c^{\op}\iso \End_{\U_c}\,Q_c$.

 \subsection{Kazhdan filtrations} Given an algebra
$A$ with an ascending $\Z$-filtration $F_\idot A$, one puts  $\gr_F A:=\bplus_{n\in\Z}\, F_n
A/F_{n-1}A$, an
associated graded algebra, resp.  $\ree{F} A:=\bplus_{n\in\Z}\, F_n
A$, the
Rees algebra.
From now on, we assume that $\gr_F A$ is  a finitely generated commutative algebra.

Let $V$ be a left $A$ module equipped with an ascending $\Z$-filtration
$F_\idot V$ which is compatible with the one on $A$.
Then,  $\ree{F}V:=\bplus_{n\in\Z}\, F_nV$ acquires the
structure of a  left $\ree{F} A$-module. The filtration on $V$ is called
{\em good} if that module  is a finitely generated.
In such a case, $\gr_F V:=\bplus_{n\in\Z}\, F_nV/F_{n-1}V$  is a
finitely generated 
$\gr_F A$-module, and
the {\em support} of $\gr_F V$ is a closed subset
$\supp(\gr_F V)\sset \Spec(\gr_F A)$ (equipped with reduced scheme
structure).

The following standard result is well known, cf.
 \cite{ABO}, Theorem 1.8 and Proposition 2.6.

\begin{Lem}\label{good} Let  $F_\idot A$ be a  $\Z$-filtration on an
algebra $A$ such that $\ree{F}A$ is  both left and
 right  noetherian, and $\gr_F A$ is commutative. Then, for any   left $A$-module $V$,
we have
\vskip 2pt

\vi All good filtrations on $V$ are equivalent to each other
and the set $\var V:=\supp(\gr_F V)$ is independent of the choice of such a filtration.

\vii For any good filtration $F_\idot V$ and any  $A$-submodule $N\sset V$,
the induced filtration
$F_\idot N:=N\cap F_\idot V$
is a good filtration on $N$.\qed
\end{Lem}

Given a vector space $V$ equipped with an  ascending
$\Z$-filtration $F_\idot V$ and with a  direct sum decomposition
$V=\bplus_{a\in\C}\ V\langle a\rangle$, one defines
an associated {\em  Kazhdan filtration},
a new  ascending $\Z$-filtration on $V$, as follows
\beq{FK} 
\kk_n V:=\sum_{a+2j\leq n}\  \big(F_jV\  \cap\   V\langle a\rangle\big).
\eeq

Let $\ug=\bplus_{i\in\Z}\ \ug\langle i\rangle$ be the $\Z$-grading induced by
{\em some}  Lie algebra grading $\g=\bplus_{i\in\Z}\ \g\langle
i\rangle$.
We may take $V:=\ug$ and let $F_j V:=\U^{\leq j}\g,\
j=0,1,\ldots,$ be the canonical ascending PBW filtration on the
enveloping algebra.
 Then, formula \eqref{FK} gives an associated  Kazhdan filtration
$\kk_\idot\ug$, on $\ug$. This
filtration is multiplicative, i.e.,
one has $\kk_i\ug\cdot\kk_j\ug\sset\kk_{i+j}\ug,\ \forall i,j$.

For any $i\in\Z$, we have $\g\langle i\rangle\in \kk_{i+2}\ug$.
We see  as in \cite[\S4.2]{GG} that the identity map $\g\to\g$
extends to  a Poisson algebra isomorphism  $\gr_\kk\ug\cong\sym\g$.
Further, it follows easily that the algebra $\ree{\kk}\ug$ is isomorphic to
a quotient of $T\g\o\C[t]$ by the two-sided ideal generated
by the elements $x\o y-y\o x - [x,y]\o t^2,\ x,y\in\g$.
Thus,
 $\ree{\kk}\ug$,  is a finitely generated algebra. Moreover, $\ree{\kk}\ug$,
viewed as an algebra {\em without grading},
is {\em independent}  of the grading on $\g$ used in \eqref{FK} (for $V=\ug$).
For the trivial grading $\g=\g\langle0\rangle$, the 
filtration $\kk_\idot\ug$ is clearly non-negative, and the corresponding
algebra
 $\ree{\kk}\ug$ is easily seen to be  both left and
 right
noetherian. It follows that  $\ree{\kk}\ug$ is  both left and
 right noetherian for {\em any} Lie algebra $\Z$-grading on $\g$.
In particular, Lemma \ref{good} applies.

From now on, we let $\kk_\idot\ug$ be
the  Kazhdan filtration associated with the grading \eqref{weight}
by formula \eqref{FK}.
Given a left $\ug$-module $V$, one may consider ascending
filtrations $F_\idot V$ which are compatible with the PBW filtration on $\ug,$
in the sense that $\U^{\leq i}\g\cdot F_j V\sset F_{i+j}V$
holds
for any $i,j\in\Z$.
One may also consider  filtrations  $\kk_\idot V$
which are compatible with the above defined Kazhdan 
filtration on $\ug$, to  be  referred to as
Kazhdan filtrations on $V$.

Assume, in addition, that  the $h$-action on $V$ is locally finite.
Then  there is a direct sum decomposition
$V=\bplus_{a\in\C}\ V\langle a\rangle$ where,
for any $a\in\C$, one defines a generalized $a$-eigenspace by the formula
$V\langle a\rangle:=\{v\in V\mid \exists N=N(v): \; (h-a)^Nv=0\}.$
For any  $h$-stable ascending filtration   $F_\idot V$ which is
compatible with the PBW filtration on $\ug$,
formula \eqref{FK}  gives an associated
   Kazhdan filtration $\kk_\idot V$, on $V$. Clearly,
 one has $F_jV=\bplus_{a\in\C}\ \big(F_jV\  \cap\ V\langle
a\rangle\big).$  Using this,  one obtains by an appropriate `re-grading procedure'
a canonical  isomorphism of  $\sym\g$-modules, resp. $\ree{\kk}\ug$-modules,
(that does {\em not} necessarily respect the natural gradings):
\beq{F=K}
\gr_FV\cong\gr_\kk V,\quad\text{resp.}\quad\ree{F}V\cong\ree{\kk}V.
\eeq

\begin{Cor}\label{FK_cor} Let  $V$ be a finitely generated $\ug$-module
such that  the $h$-action on $V$ is locally finite. Then, one has

\vi For any good  $h$-stable filtration  $F_\idot V,$ on $V$,
formula \eqref{FK} gives a good and separated  Kazhdan filtration on $V$.

\vii Any  good Kazhdan filtration  $\kk'_\idot V$, on $V$,
 is equivalent to
 \eqref{FK}. Hence, it is a separated filtration and,
 in $\g^*$, one has $\var_{\kk'} V=\var_F V$ (set theoretic equality).
\end{Cor}
\begin{proof}  Let  $F_\idot V$ be a good
 $h$-stable  filtration. Clearly, it is bounded from below,
hence, separated. It follows that
$\ree{F}V$ is a finitely generated $\ree{F}\ug$-module and,
moreover, we have $\cap_{j\in\Z}\ t^j\cdot\ree{F}V=0$.
Thus,  from \eqref{F=K}  we deduce
that $\ree{\kk}V$ is a finitely generated $\ree{\kk}\ug$-module and,
moreover, we have $\cap_{j\in\Z}\ t^j\cdot\ree{\kk}V=0$.
This yields (i). Part (ii) is now a consequence of Lemma \ref{good}.
\end{proof}

From now on, in the setting of Corollary \ref{FK_cor},
we will use simplified notation $\var V$ for  $\var_{\kk'} V=\var_F V.$ 
 
\subsection{Whittaker functors} Recall  the space
 $Q_{c}=\U_{c}/\U_{c}\m_\chi$, which is an
$(\U_{c},A_{c})$-bimodule.

Associated with
 any left $\ug$-module, resp. $\U_c$-module, $V$, is its {\em Whittaker subspace}:
\beq{whdef}\Hom_{\U_c}(Q_c, V)=\Wh^\m(V) :=
\{ v \in V \mmid mv=\chi(m)v\,,\,\forall m \in \m \}.
\eeq

The right
 $A_{c}$-action
on $Q_{c}$ gives the space $\Wh^\m V$ a structure of {\em left}
$A_{c}$-module. Similarly, for any {\em right} $\U_{c}$-module $N$, the right
 $A_{c}$-action
on $Q_{c}$ gives the coinvariant space
$N/N\m_\chi=N\o_{\U_{c}}\,Q_{c}$ a structure of {\em right}
$A_{c}$-module.

Below, we are also  going to use a version of the  Whittaker functor for
 bimodules.
Given an $(\U_{c'},\U_{c})$-bimodule $K$, the subspace
$K\m_\chi\sset K$ is stable under the
 adjoint $\m$-action on
$K$. Hence, the latter action descends to a well defined $\ad\m$-action
on $K/K\m_\chi$. It is clear that, for any
$x\in K/K\m_\chi$ and $m\in \m$, one has
$mx=\ad m(x) +\chi(m)\cdot x$. We see in particular that
 the
 adjoint $\m$-action and  the above defined right
$A_{c}$-action on $K/K\m_\chi$ commute.

Finally, we put
\beq{whb}\whb(K):=\Hom_{\U_{c'}}(Q_{c'},\ K\o_{\U_{c}}\,Q_{c})=
(K/K\m_\chi)^{\ad \m}.
\eeq

The right $A_{c'}$-action on $Q_{c'}$ and the  right
 $A_{c}$-action
on $Q_{c}$ make $\whb(K)$ an $(A_{c'},A_{c})$-bimodule.

\begin{defn}
 Let $\catu$ be the
abelian
category of finitely generated
$\U_c$-modules $V$ satisfying the following condition:
{\em for any
$v\in V$ there exists an integer  $n=n(v)\gg0$ such that,
we have} $(m_1\cdot m_2\cdot\ldots\cdot m_n)v=0,\
\forall m_1,\ldots,m_n\in\m_\chi$.
\end{defn}

Objects of the category $\catu$ are called Whittaker modules.
It is clear that $Q_c$ is a Whittaker module.
We let  $\kk_\idot\U_c$, resp. $\kk_\idot Q_c$, be the quotient filtration
 induced by the Kazhdan filtration on $\ug.$
The  Kazhdan filtration 
 on $Q_c$ induces, by restriction,
 an algebra filtration  $\kk_\idot A_c$, on $A_c\sset Q_c$.
Note that, unlike the case of $\kk_\idot\U_c$,
 the filtration  $\kk_\idot Q_c$, hence also  $\kk_\idot A_c$,
is {\em non-negative}.

The proof of the following result repeats the proof of
\cite[Proposition ~5.2]{GG}.
\begin{Prop}\label{pp1} For any $c\in\Specm\zg,$  one has a
graded Poisson algebra isomorphism $\gr_\kk \AA_c\ccong \C[\sn],$
cf. \eqref{put2}. Furthermore, 
 one has a graded $(\gr_\kk\U_c,\gr_\kk A_c)$-bimodule
isomorphism $\gr_\kk Q_c\cong S_o/S_o\m_\chi=$, cf.
\eqref{put1}.
\qed
\end{Prop}

From this proposition, using Proposition \ref{kost} and a result of
Bjork \cite{Bj},
we deduce
\begin{Cor}\label{gorenst}
The algebra $A_c$ is Cohen-Macaulay and Auslander-Gorenstein.\qed
\end{Cor}

Let  $\kk_\idot V$ be a good Kazhdan filtration on an object 
$V\in\catu$. Using that the Kazhdan filtration on $Q_c$ is nonnegative,
one proves that the  filtration $\kk_\idot V$
 is bounded from below.
 If, in addition, the filtration $\kk_\idot V$ is
$\m$-stable
then $\gr_\kk V$ is an $M$-equivariant $\C[\N]$-module
such that the
 $\C[\N]$-action  factors through
an action of the algebra $\C[\N\cap (\chi+\m^\perp)]$.
In particular, we have $\var V\sset \N\cap
(\chi+\m^\perp).$

One has a graded algebra isomorphism
$\C[\N\cap (\chi+\m^\perp)]\cong \C[M]\o \C[\sn]$ that results from
the scheme isomorphism of Proposition \ref{ggnil}(iii).
Here, the grading on $\C[M]$ is the weight grading with respect to the
adjoint
action of the 1-parameter subgroup
$t\mto \gamma_{t\inv}$.
Hence, we get $\gr_\kk Q_{c}=\C[M]\o \C[\sn]$, by
 Proposition \ref{pp1}.

For any noetherian algebra $B$, let $\modu B$ denote the category of
finitely generated left
$B$-modules.
Also, in part (ii) of the proposition below, given a  filtration on a 
$A_{c}$-module $N$, we equip  $Q_{c}\o_{A_{c}} N$ with the
tensor product filtration using the Kazhdan filtration on $Q_{c}$.

\begin{Prop}\label{skr} \vi The functor $\Wh^\m:\ \catu \to \modu{A_{c}}$ is an
equivalence.
\vs

\vii For any good filtration on a 
$A_{c}$-module $N$, the natural map
$\C[M]\o \gr N\iso\gr_\kk(Q_{c}\o_{A_{c}} ~N)$ yields
a graded $\C[\N]$-module isomorphism.
Furthermore, the functor $N\mto Q_{c}\o_{A_{c}} N$ provides 
a quasi-inverse to the equivalence in (i).
\vs

\viii For any good $\m$-stable Kazhdan filtration $\kk_\idot V$, on 
$V\in\catu$,  there are canonical isomorphisms
$\gr(\Wh^\m V)\cong(\gr_\kk V)^M\cong (\gr_\kk V)|_\sn,$ of  graded $\C[\sn]$-modules.\qed
\end{Prop}

The equivalences in parts (i)-(ii) of the above proposition  are due to Skryabin,
\cite{skryabin},
and the graded isomorphisms in
parts (ii)-(iii) are immediate consequences of the results of \cite{GG}.

\section{Weak Harish-Chandra bimodules}
\subsection{}\label{hccat} Let $B$ and $B'$ be an arbitrary pair
of nonnegatively
filtered algebras such that $\gr B$ and $\gr B'$, the corresponding
associated graded algebras, are finitely generated commutative
algebras {\em isomorphic to each other}. Thus, there is a well defined
 subset $\Delta\sset \Spec(\gr B)\times\Spec(\gr B'),$
the diagonal. 

Associated with any finitely generated   $(B,B')$-bimodule
$K$, viewed as a left $B\o (B')^{\op}$-module, is its characteristic
variety $\var K\sset \Spec(\gr B)\times\Spec(\gr B').$ 

\begin{defn}\label{whc_def} A  finitely
 generated  $(B,B')$-bimodule $K$ is called a
{\em weak Harish-Chandra} (\whc)  bimodule
if, set theoretically, one has $\var K\sset \Delta.$
\end{defn}

It is straightforward to show the following
\begin{Prop} \vi Any \whc   $\en(B,B')$-bimodule,
viewed either as a left
$B$-module, or as a right $B'$-module, is finitely
generated.

\vii The category $\WHC(B,B')$, of \whc\; bimodules,
is an abelian category. \qed
\end{Prop}

Given a closed subset $Z\sset\Spec(\gr B)=\Spec(\gr B')$,
let $\modu_ZB,$ resp. $\modu_ZB'$ be the category of
finitely generated left $B$-modules, resp. $B'$-modules, $K$
such that $\var K\sset Z$.
One similarly defines
$\WHC_Z(B,B')$ to be the Serre subcategory
of $\WHC(B,B')$ formed by the \whc\; bimodules
$K$ such that $\var K\sset Z$. 

Tensor product over $B'$  gives a bi-functor
\beq{convolution}
\WHC(B,B') \times \modu_ZB\too \modu_ZB',\quad
K\times V\mto K\o_{B'} V.
\eeq

We have, in particular, the category $\WHC(\ug,\ug)$, where the enveloping algebra $\ug$
is equipped with the PBW filtration, not
with Kazhdan filtration.
Similarly, for any ${c'},{c}\in\Specm\mathfrak{Z}\g$, one has
 the category 
$\WHC(\U_{c'},\U_{c})$, resp. $\WHC(A_{c'},A_{c})$.

A  finitely generated  $(\ug,\ug)$-bimodule,
resp. $(\U_{c'},\U_{c})$-bimodule, $K$ such that
the adjoint $\g$-action $\ad a: v\mto av-va,$ on $K$, is locally
finite is called a Harish-Chandra bimodule.
It is clear that any  Harish-Chandra bimodule
is  a weak Harish-Chandra bimodule. However, the converse
 is not true, in general, cf. section \ref{rs_sec}.
We write $\hc(\U_{c'},\U_{c})$ for the abelian category
of  Harish-Chandra  $(\U_{c'},\U_{c})$-bimodules.

Let $K$ be a
 Harish-Chandra  $(\U_{c'},\U_{c})$-bimodule
and let $F_\idot K$ be a good  $\ad\g$-stable filtration compatible
with  the tensor product  of PBW filtrations on $\U_{c'}$
and on $\U_{c}^{\op}$.
The $\ad h$-action on $K$ being locally finite,
one can use formula \eqref{FK} to define an
associated Kazhdan filtration $\kk_\idot K$, on $K$.
The latter induces a quotient filtration on
$K/K\mc$ which gives, by restriction,
a filtration $\kk_\idot(\Wh^\m_\m K)$, on $\Wh^\m_\m K$.
It is easy to see that this filtration is compatible with
the algebra filtration on $A_{c'}\o A_c^{\op}.$

Recall the notation $\O=\Ad G(\chi)$.
Let $\N^\Diamond$ be the union of all nilpotent $\Ad G$-orbits $\O'\sset\g^*$ such
that $\O\not\subset\overline{\O'}$. Thus, $\N^\Diamond\sset\N$ is
a closed $\Ad G$-stable subset such that $\N^\Diamond\cap\sn=\emptyset.$

The first main result of the paper is the following

\begin{Thm}\label{main_thm}  Let
$K\in \hc(\U_{c'},\U_{c})$. Then,
the following holds:\vs

\vi For any good $\ad\g$-stable filtration on $K$,
the associated filtration  $\kk_\idot(\Wh^\m_\m K)$ is good; moreover, one has
a  graded isomorphism  of $\gr A_{c'}$- and  $\gr A_c$-modules,
$\gr(\Wh^\m_\m K)\cong (\gr K)|_\sn$; in particular,
$\var(\Wh^\m_\m K)=\sn\cap \var K.$\vs

\vii The functor $K\o_{\U_{c}}(-)$ takes Whittaker modules to
 Whittaker modules, and there is a
natural isomorphism of
functors
that makes the following diagram commute
\beq{comm_phi}
\xymatrix{
\catu\ar[d]_<>(0.5){K\o_{\U_{c}}(-)}
\ar[rrr]^<>(0.5){\Wh^\m}
&&&\modu{\AA_{c}}\ar[d]^<>(0.5){\Wh^\m_\m K\o_{\AA_{{c}}}(-)}\\
{{(\U_{{c'}},\mc)\mbox{-}\operatorname{mod}}}
\ar[rrr]^<>(0.5){\Wh^\m}
&&&\modu{\AA_{{c'}}}
}
\eeq

 \viii The assignment $K\mto \Wh^\m_\m K$ induces
a {\em  faithfull exact}  functor
$$\hc(\U_{c'},\U_{c})/\hc_{\N^\Diamond}(\U_{c'},\U_{c})
\
\to\
\WHC(A_{c'},A_{c}).$$
\end{Thm}

The proof of Theorem \ref{main_thm} occupies subsections
\S\S\ref{vanish_sec}-\ref{pf_sec}.
From part (i) of the theorem, one immediately obtains

\begin{Cor}\label{fin} For $K\in \hc(\U_{c'},\U_{c})$, we have:
\begin{align*}
\O\subset\var K\en\;&\Longleftrightarrow\en\;\Wh^\m_\m K\neq0;\\
\overline{\O}=\var K\en\;&\Longleftrightarrow\en\;\dim(\Wh^\m_\m K)
<\infty.\tag*{$\Box$}
\end{align*}
\end{Cor}

The following direct consequence of Theorem \ref{main_thm}(ii)
says that $\Wh^\m_\m$ is a monoidal functor.
\begin{Cor} There is a functorial isomorphism of $(A_{c'}, A_{c''})$-bimodules
$$\Wh^\m_\m(K\o_{\U_c} K')\ \iso\ (\Wh^\m_\m K)\o_{A_c}(\Wh^\m_\m K'),
\qquad
K\in \hc(\U_{c'},\U_{c}),\ K'\in \hc(\U_c,\U_{c''}).\eqno\Box$$
\end{Cor}

\subsection{}\label{Isec} 
We are going to construct a functor
from $(A_{c'},A_{c})$-bimodules to $(\U_{c'},\U_{c})$-bimodules
as follows.
Let $N$ be an $(A_{c'},A_{c})$-bimodule. 
Then, the space $Q_{c'}\o_{A_{c'}}N$ has an obvious structure
of  $(\U_{c'}, A_{c})$-bimodule.
Let
$\dis\wt{I}(N):=\Hom_{A_c}(Q_{c},\ Q_{c'}\o_{A_{c'}} N)$
be the space of linear maps $Q_{c}\to Q_{c'}\o_{A_{c'}} N$ which
commute with the right $A_c$-action. 
The left $\U_{c}$-action on $Q_{c}$ and the left $\U_{c'}$-action on
$Q_{c'}\o_{A_{c'}}\,N$
make this space
 an $(\U_{c'},\U_{c})$-bimodule.
We put
\beq{I}
I(N):=\Hom^{\text{fin}}_{A_c}(Q_{c},\ Q_{c'}\o_{A_{c'}} N)\
\sset\ \wt{I}(N),
\eeq
 the subspace  of $\dis\wt{I}(N)$ formed by   $\ad\g$-locally finite
elements.

Part (i) of the theorem below will be proved later, in section
 \ref{strong}. It is stated here for reference purposes.

\begin{Thm}\label{adj} \vi Any object of $\WHC(A_{c'},A_{c})$
has
finite length.

\vii The assignment $N\mto I(N)$ gives a functor
 $I:\ \WHC(A_{c'},A_{c})\to \hc(\U_{c'},\U_{c}),$
which is a  right adjoint of $\Wh^\m_\m$.
\end{Thm}
\begin{proof} 
We begin the proof of (ii) by showing the adjunction property. The latter says that,
for any
$K\in\hc(\U_{c'},\U_{c}),\ N\in\WHC(A_{c'},A_{c}),$ there is  a bifunctorial isomorphism
\beq{Iadj}
\Hom_{\bimod{A_{c'},A_{c}}}(\Wh^\m_\m K,\ N)\
\cong\
\Hom_{\bimod{\U_{c'},\U_{c}}}(K,\ I(N)).
\eeq

To prove this,  using the definition of $\wt{I}(N),$ we compute 
\begin{align}
\Hom_{\bimod{\U_{c'},\U_{c}}}(K,\ \wt{I}(N))&=
\Hom_{\bimod{\U_{c'},\U_{c}}}\big(K,\ \Hom_{\Rmod{A_c}}(Q_{c},\
Q_{c'}\o_{A_{c'}} N)\big)\nonumber\\
&=\Hom_{\bimod{\U_{c'},\U_{c}\o A_c}}(K\o Q_{c},\
Q_{c'}\o_{A_{c'}} N)\nonumber\\
&=\Hom_{\bimod{\U_{c'}, A_c}}(K\o_{\U_{c}} Q_{c},\
Q_{c'}\o_{A_{c'}} N)\label{qqq}\\
&=\Hom_{\bimod{\U_{c'}, A_c}}(K/K\mc,\
Q_{c'}\o_{A_{c'}} N).\nonumber
\end{align}

Now, both  $K/K\mc$ and
$Q_{c'}\o_{A_{c'}} N,$ viewed as left $\U_{c'}$-modules, are objects of
$\catu$. For these objects, we have $\Wh^\m(K/K\mc)=\Wh^\m_\m K$
and $\Wh^\m(Q_{c'}\o_{A_{c'}} N)=N$.
Hence, by the Skryabin equivalence, we get
$\Hom_{\U_{c'}}(K/K\mc,\
Q_{c'}\o_{A_{c'}} N)=
\Hom_{A_{c'}}(\Wh^\m_\m K, N).$

We deduce an isomorphism
$$\Hom_{\bimod{\U_{c'}, A_c}}(K/K\mc,\
Q_{c'}\o_{A_{c'}} N)=\Hom_{\bimod{A_{c'},A_c}}(\Wh^\m_\m K, N).$$
Thus, from \eqref{qqq}, we obtain
\beq{II}\Hom_{\bimod{\U_{c'},\U_{c}}}(K,\ \wt{I}(N))=
\Hom_{\bimod{A_{c'},A_c}}(\Wh^\m_\m K, N).
\eeq

Observe next  that since the $\ad\g$-action on $K$ is locally finite
the imbedding $I(N)\into \wt{I}(N)$ induces a bijection
$\Hom_{\bimod{\U_{c'},\U_{c}}}(K,\ I(N))\iso
\Hom_{\bimod{\U_{c'},\U_{c}}}(K,\ \wt{I}(N)).$
Hence, in the  $\Hom$-space on the left of \eqref{II}, 
one may replace $\wt{I}(N)$ by $I(N)$.
The resulting isomorphism yields \eqref{Iadj}.

To complete  the proof of part (ii) we must show that
$I(N)$ is a finitely generated  $(\U_{c'},\ug)$-bimodule.
To this end, we need to enlarge the
category $\hc(\U_{c'},\U_{c})$ as follows.
Recall first that
 $\I_c\sset\zg$ denotes the maximal ideal in the center of
the enveloping algebra $\ug$. Let $K$ be a
finitely generated $(\U_{c'},\ug)$-bimodule.
We say that $K$ is a Harish-Chandra  $(\U_{c'},\wh\U_{c})$-bimodule
if the adjoint $\g$-action on $K$ is locally finite and, moreover,
there exists an large enough integer $\ell=\ell(K)\gg0$ such that
$K$ is annihilated by the right action of the ideal $(\I_c)^\ell$,
that is, we have $K\cdot(\I_c)^\ell=0.$
Let $\hc(\U_{c'},\wh\U_{c})$ be the full subcategory
of $\bimod{\U_{c'},\ug}$ whose objects are 
 Harish-Chandra  $(\U_{c'},\wh\U_{c})$-bimodules.
The structure of the category  $\hc(\U_{c'},\wh\U_{c})$
has been analyzed by Bernstein and Gelfand  \cite{BGe}.
It turns out that any  Harish-Chandra  $(\U_{c'},\wh\U_{c})$-bimodule
has finite length. Furthermore, 
it was shown in {\em loc cit} that the category $\hc(\U_{c'},\wh\U_{c})$
has enough projectives
and there are only finitely many nonisomorphic
indecomposable projectives $P_j, \ j=1,\ldots, m,$ say.

Next, we observe that the algebra
$A$, introduced in Remark \ref{A},
 comes equipped with a natural ascending filtration
such that one has $\gr A=\C[\s]$. 
A finitely generated $(A_{c'},A)$-bimodule $N$ will be called
a weak  Harish-Chandra  $(A_{c'},A)$-bimodule if 
one has $\var N\sset \sn\sset\sn\times\s$, where
$\sn\into\sn\times\s$ is the diagonal imbedding.
We let $\WHC(A_{c'},\wh A_{c})$ denote the full subcategory
of the category $\bimod{A_{c'},A}$ whose objects are
 weak  Harish-Chandra  $(A_{c'},A)$-bimodules $N$ such that
$N\cdot(\IJ_c)^\ell=0$ holds for a  large enough integer $\ell=\ell(N)\gg0$.
As we will see later, the arguments 
of  \S\ref{strong} can be used to show that
any object of the category  $\WHC(A_{c'},\wh A_{c})$
has finite length.

Clearly, the category $\hc(\U_{c'},\U_{c})$ may be
viewed as a full subcategory in  $\hc(\U_{c'},\wh\U_{c})$,
resp. the category $\WHC(A_{c'},A_{c})$ may be
viewed as a full subcategory in  $\WHC(A_{c'},\wh A_{c})$.
It is straightforward to extend our earlier definitions and introduce
a functor $\Wh_\m^\m:\ 
\hc(\U_{c'},\wh\U_{c})$ $\to\WHC(A_{c'},\wh A_{c})$.
One also defines a functor $I$ in the opposite direction
such that an analogue of  formula \eqref{Iadj} holds.

We are now ready to complete the proof of Theorem \ref{adj}(ii) by
showing that
$I(N)$ is a finitely generated $(\U_{c'},\ug)$-bimodule,
for any $N\in \WHC(A_{c'},\wh A_{c})$.
It suffices to show, in view of the results
of \cite{BGe} cited above, that,  for each $j=1,\ldots, m,$ the vector space
$\Hom_{\bimod{\U_{c'},\wh\U_{c}}}(P_j, \ I(N))$ has finite dimension.
To see this, we use the analogue of formula \eqref{Iadj}, which yields
\beq{ddd}
\dim\Hom_{\bimod{\U_{c'},\wh\U_{c}}}(P_j, \ I(N)
=
\dim\Hom_{\bimod{A_{c'},\wh A_{c}}}(\Wh^\m_\m P_j,\ N).
\eeq
But, we know that $\Wh^\m_\m P_j\in \WHC(A_{c'},\wh A_{c})$
and that the object $N\in \WHC(A_{c'},\wh A_{c})$ has finite length.
It follows that the dimension on the right of  formula \eqref{ddd}
is finite, and we are done.
\end{proof}

\subsection{Homology vanishing}\label{vanish_sec}
 Recall the Lie subalgebra $\m_\chi\sset\ug.$
The Kazhdan 
filtration on $\ug$ restricts to a
filtration on  $\m_\chi$. The latter induces
 an ascending
filtration $\kk_j(\wedge^\hdot\m_\chi),\,j\geq0$, on
 the exterior algebra of $\m_\chi$.

Given a {\em right} 
$\ug$-module $V$ equipped
with a Kazhdan filtration $\kk_\idot V$, we form a tensor product
$C:=V\o \wedge^\hdot\m_\chi$, and let $\kk_nC=
\sum_{n=i+j} \kk_i V \o \kk_j(\wedge^\hdot\m_\chi)$
be the tensor product filtration.

We view $V$ as a $\m_\chi$-module, and write $H_\idot(\m_\chi, V)$
for the corresponding Lie algebra homology with coefficients in $V$.
The latter may be computed by means of the
 complex $(C, \partial)$, where
$\partial: C\to C$ is the standard
Chevalley-Eilenberg differential, of degree $(-1)$. It is immediate to check that
the differential  respects the
filtration $\kk_\idot C$, making $C$ a filtered complex.

 Write $B\sset Z\sset C$ for the subspaces
of boundaries and cycles of the complex $C$, respectively.
Thus, we have $H_\idot(\m_\chi, V)=H_\idot(C)=Z/B$.
The filtration on $C$ induces, by restriction, a
filtration 
$\kk_p Z:=Z\cap\kk_p C,$  on the space of cycles.
The latter filtration induces a quotient filtration on homology.
Explicitly,    the
induced filtration on homology is given by
\beq{filtH}
\kk_p H(C):=(B+\kk_p Z)/B=\im[H(\kk_p C)\to H(C)],\qquad p\in\Z.
\eeq

There is  an associated standard
 spectral sequence
 with 0-th term, cf. \cite{CE}, Ch.15,
\beq{EE}
E^0_{p,q}=\wedge^{p+q}\m_\chi\o\gr_{-q}
V.
\eeq

 Recall the local algebra $\oo_\chi$ introduced
above Corollary \ref{flat}.
The lemma below is a slight generalization of a result due to Holland,
cf. \cite{H}, \S 2.4.
\begin{Lem} \label{holland} Let $V$ be a right $\ug$-module  equipped
with a {\em good} Kazhdan filtration  $\kk_\idot V$. Assume that
the localization of $\gr_\kk V$  at $\chi+\m^\perp$ is a flat
$\varpi^\hdot\oo_\chi$-module and, moreover,  the
induced filtration on $C$, the Chevalley-Eilenberg complex,
is {\em convergent} in the
sense of \cite{CE}, p. 321.

Then, for any $j>0$, we have $ H_j(\m_\chi,\gr_\kk V)=0$ and
 $ H_j(\m_\chi, V)=0.$
Moreover,
the natural projection yields a canonical graded $\C[\g^*]$-module  isomorphism
$$\gr_\kk V/(\gr_\kk V)\m_\chi\iso\gr_\kk (V/V\m_\chi)=\gr_\kk
H_0(\m_\chi, V).
$$
\end{Lem}
\begin{proof} We recall various standard objects associated
with the spectral sequence of a filtered complex. We follow 
\cite{CE}, Ch. 15, \S1-2 closely, except that we use homological, rather than 
{\em co}homological notation. 

First of all, for each $p\in\Z$, one defines a pair of vector spaces
$Z^\infty_p$ and 
$B^\infty_p$, in such a way that one has
\beq{grH1}
\gr_p  H_\idot(\m_\chi, V)=\gr_p  H_\idot(C)\cong Z^\infty_p/B^\infty_p.
\eeq

Further, 
one defines a chain of vector spaces, cf. \cite{CE}, p.317, 
$$\ldots\sset B^r_p\sset B^{r+1}_p\sset\ldots\sset
B^\infty_p\sset Z^\infty_p\sset\ldots
Z^{r+1}_p\sset Z^r_p\sset\ldots.
$$

Precise definitions of these objects are not important for our purposes,
they are given e.g. in  \cite{CE}, Ch. 15, \S1. What {\em is} important for
us
is  that the definitions imply $B^\infty_p=\cup_{r\geq 0}\ B^r_p$.
The assumption of the lemma that the filtration $\kk_\idot C$
be convergent means that, in addition, one  has, see \cite{CE}, ch. 15, \S2,
\beq{convergent}
Z^\infty_p=\cap_{r\geq 0}\ Z^r_p\quad\text{and}\quad
\cap_{r\geq 0}\ \kk_rH_p(C)=0, \quad\forall p\in\Z.
\eeq

Now, to prove the lemma, we observe that
the  zeroth page $(E^0, d^0)$ of the spectral
sequence \eqref{EE} may be identified with
the Koszul complex associated with the subscheme
$\varpi\inv(\chi_\m)\sset \g^*$.
Thus, the assumption of the lemma that the
localization of $\gr_\kk V$  at $\chi+\m^\perp$  be a flat
$\varpi^\hdot\oo_\chi$-module forces 
the Koszul complex be acyclic in positive degrees.
Therefore, the spectral sequence  degenerates at $E^1$.

The degeneration implies that, for any $p\in\Z$ and $r\geq 1$, one has $Z^r_p=Z^1_p$
and $B^r_p=B^1_p$.
Therefore, we have  $B^1_p=B^\infty_p$, and using the first equation in \eqref{convergent},
we get $Z^1_p=Z^\infty_p.$ Hence,
$H(E^0,d^0)=Z^1/B^1=Z^\infty_p/B^\infty_p=
\gr_p H_\idot(\m_\chi, V),$ by \eqref{grH1}.
 We conclude that one has $\gr_\kk H_0(\m_\chi, V)=
(\gr_k V)/(\gr_k V)\m_\chi$ and,
moreover,  
$\gr_\kk H_j(\m_\chi, V)=0$ for any $j>0$.
Finally, thanks to the second equation in \eqref{convergent},
we have
$\gr_\kk H_j(\m_\chi, V)=0\en\Rightarrow\en H_j(\m_\chi, V)=0,$
and the lemma is proved.
\end{proof}

It is important to observe that the filtration $\kk_\idot C$, on $C$,
gives rise to {\em two} natural filtrations on the subspace
$B= \partial C$, of the boundaries. These two filtrations are defined as follows
\beq{KB} \kk_\idot B\ :=\ \partial(\kk_\idot C),\quad\text{resp.}
\quad
\kk'_\idot B\ :=\ B\ \cap\  \kk_\idot C.
\eeq

It is clear that on has $\kk_\idot B\sset \kk_\idot' B,$ but this inclusion need not
be an equality, in general. We have the following
result
\begin{Lem}\label{CE}  Assume the above filtrations $\kk_\idot B$ and
$\kk'_\idot B$ are {\em
equivalent},
i.e. there exists an integer $\ell\geq 1$ such that,
for all $p\in\Z$, one has $\kk'_{p-\ell} B\sset\kk_{p-1}
B$.

Then, we have $Z^\infty_p=\cap_{r\geq 0}\ Z^r_p,\ \forall p\in\Z$, i.e. the first equation
in
\eqref{convergent}, holds.
\end{Lem}
\begin{proof} 
For any $\ell\geq 1,$ we have $\kk_{\idot-\ell} C\sset\kk_{\idot-1}C$,
hence, one gets an obvious imbedding 
$B\ \cap\  \kk_{\idot-\ell} C\into Z\cap \kk_{\idot-1}C$.
Using the definitions of  various filtrations introduced above,
this  imbedding  may be rewritten as follows
$\kk'_{\idot-\ell} B\into \kk_{\idot-1}Z$. 
We may further compose the imbedding with a projection to homology to
obtain the following composite
\beq{ell1}
\delta:\ 
\kk'_{\idot-\ell} B\into \kk_{\idot-1}Z\onto
\kk_{\idot-1}Z/\kk_{\idot-1}B=\kk_{\idot-1}Z/\pa(\kk_{\idot-1}C)=
H(\kk_{\idot-1}C).
\eeq

Next, we fix $p\in\Z$ and consider the complex $\kk_pC/\kk_{p-\ell}C$.
By definition, we have
$$
H(\kk_pC/\kk_{p-\ell}C)= \frac{\{z\in\kk_pC\mid\pa(z)\in
 \kk_{p-\ell}C\}}{\kk_{p-\ell}Z + \pa(\kk_pC)}=
 \frac{\{z\in\kk_pC\mid\pa(z)\in
\kk'_{p-\ell}B\}}{\kk_{p-\ell}Z + \pa(\kk_pC)}.
$$

The differential $\pa$ clearly annihilates the space
$\kk_{p-\ell}Z + \pa(\kk_pC)$, in the denominator of the rightmost
term. Therefore,
we see that applying the differential $\pa$ to the numerator
of that term yields  a well defined map
$$\pa: \ 
H(\kk_pC/\kk_{p-\ell}C)\to \kk'_{p-\ell}B,\quad z\mto\pa(z).$$

Thus,  we have constructed the following diagram
$$
\xymatrix{
H(\kk_\idot C/\kk_{\idot-\ell}C)\ \ar[r]^<>(0.5){\pa}&
\  \kk'_{\idot-\ell} B\ \ar[r]^<>(0.5){\delta}&
\ H(\kk_{\idot-1}C).
}
$$

Now, let $\ell$ be such that the assumption of the lemma
holds, so that we have
$\kk'_{\idot-\ell} B\sset \kk_{\idot-1} B$. Then the corresponding
 map $\delta$, in \eqref{ell1}, clearly vanishes.
It follows that the composite map $\delta\ccirc\pa$ vanishes as well,
and we have
$$\im\big[\delta\ccirc\pa:\ H(\kk_\idot C/\kk_{\idot-\ell}C)\to
H(\kk_{\idot-1}C)\big]=0.
$$

The last equation insures that we are in a position to apply a criterion
given in \cite[ch. ~15, Proposition 2.1]{CE}. Applying that criterion
yields the statement of the lemma.
\end{proof}
\subsection{Proof of Theorem \ref{main_thm}}\label{pf_sec}
Throughout this subsection, we
fix $K\in \hc(\U_{c'},\U_{c})$.

Choose a finite dimensional
 $\ad\g$-stable subspace $K_0\sset K$ that generates $K$ as a
bimodule. For any $\ell\geq 0,$ let 
$$F_\ell K:=
\sum_{i+j\leq \ell}\ \U^{\leq i}\g\cdot K_0\cdot\U^{\leq j}\g
=\U^{\leq \ell}\g\cdot K_0=K_0\cdot\U^{\leq \ell}\g.$$
In this way, one may define a good $\ad\g$-stable filtration
on $K$.

Now, let  $F_\idot K$ be an arbitrary good $\ad\g$-stable filtration
on $K$.
Let $\kk_\idot K$ be the Kazhdan filtration associated with the filtration
 $F_\idot K$ via formula \eqref{FK}. 

We first view $K$ as a right $\mc$-module.
A key step in the proof of Theorem \ref{main_thm}
is played by the following

\begin{Lem}\label{hom_van} For all $j>0$, we have  $ H_j(\m_\chi,\gr_\kk K)=0$
and $H_j(\m_\chi, K)=0$.

Furthermore, the canonical projection
$\gr_\kk K/(\gr_\kk K)\m_\chi\iso\gr_\kk(K/K\mc)$
is an isomorphism of $M$-equivariant
$\C[\N]$-modules.
\end{Lem}

\begin{proof} The result is clearly a consequence of
 Lemma \ref{holland}, provided we show that all the assumptions of
that lemma hold in our present setup.

First, we verify the assumption that, for  the above defined Kazhdan
filtration on $K$, 
the $\varpi^\hdot\oo_\chi$-module $\varpi^*\oo_\chi
\bo_{\C[\N]}\gr_\kk K$ is flat.
To this end, we note that
the construction of the filtration $F_\idot K$ insures 
 that $\gr_F K$ is an $\Ad G$-equivariant finitely generated
$\C[\N]$-module, where $\N\sset\N\times\N$
is the diagonal copy of the nilpotent variety.
 By Corollary \ref{flat}(ii), we conclude that
 the
localization of $\gr_F K$  at $\chi+\m^\perp$  
 is a flat $\varpi^\hdot\oo_\chi$-module.

Further, by \eqref{F=K}, we have
a $\C[\N]$-module isomorphism $\gr_F K\cong\gr_\kk K$.
Moreover, it is immediate from definitions that this  isomorphism
is compatible with  $\ad\m$-actions on each side.
It follows, in particular,
that $\varpi^*\oo_\chi
\bo_{\C[\N]}\gr_\kk K$ is a flat $\varpi^\hdot\oo_\chi$-module,
and we have an
$\Ad M$-equivariant $\C[\N]$-module isomorphism
$\gr_\kk K/(\gr_\kk K)\m_\chi\cong\gr_F K/(\gr_F K)\m_\chi$.

To complete the proof, we show that our filtration $\kk_\idot K$ is convergent,
i.e., both equations in \eqref{convergent} hold in the situation at hand.

To see this, we observe  that  the filtration on $K$ being $\ad\g$-stable,
it is  good as a filtration on $K$ viewed as a left $\U_{c'}$-module.
Further, the differential
in the Chevalley-Eilenberg complex $C:=\wedge^\hdot\m_\chi\o K,$
involved in  Lemma \ref{holland},
is clearly a morphism of {\em left} $\U_{c'}$-modules.
Thus, the subspace $B=\pa C\sset C$, of the boundaries
of the Chevalley-Eilenberg complex, as well as each homology group
$H_\idot(\m_\chi, K),$ acquires a natural
structure of  left $\U_{c'}$-module.
Being a subquotient of $C$, any of these 
 $\U_{c'}$-modules is finitely generated.
Hence, each of the two Kazhdan filtrations on $B$
defined in \eqref{KB},
as well as the Kazhdan filtration on 
$H_\idot(\m_\chi, K)$
defined in \eqref{filtH},
is a good filtration on the corresponding left
 $\U_{c'}$-module. It follows, in particular,
that the two  filtrations on $B$ are equivalent,
cf. Lemma \ref{good}.
Thus, the first equation in \eqref{convergent}
holds by Lemma \ref{CE}.

It remains to prove that the Kazhdan filtration on 
$H_\idot(\m_\chi, K)$
defined in \eqref{filtH} is separated.
Recall that any good Kazhdan filtration on an object of the category
$\catu$ is bounded below, hence, separated.
Thus, it suffices to show that, for any $j\geq 0$,
$H_j(\m_\chi, K)$ viewed 
as a left $\U_{c}$-module, is an object of $\catu$.

To this end, we recall that
the  Chevalley-Eilenberg complex of a right module
over a Lie algebra has a natural action of that
Lie algebra, by  the `Lie derivative'. It is well known
that the Lie derivative action 
on the Chevalley-Eilenberg complex induces the zero action on
each homology group.
Applying this to our Harish-Chandra 
$\ug$-{\em bi}module $K$,
we see  that the complex $C=K\o\wedge^\hdot\mc$ has a left
$\g$-action, defined as a tensor product of the  left $\g$-action on $K$
and the zero $\g$-action on $\wedge^\hdot\mc$.
There is also an $\mc$-action,  by  the `Lie derivative'.
The left $\g$-action and the $\mc$-action on $C$ commute, and
the difference of the left $\mc$-action and  the Lie derivative
$\mc$-action gives a
a well defined
$\mc$-action on $C$, to be called the adjoint action. 
The  adjoint $\g$-action 
on $K$  being locally finite and the Lie algebra $\mc$ being nilpotent,
 it follows easily that
 the adjoint $\mc$-action
on $C$ is locally  nilpotent.
We conclude that the left  $\mc$-action
on  $H_j(\m_\chi, K)$, induced by the left $\mc$-action on $C$,
 may be written
as a sum of a  locally  nilpotent adjoint
action and of the Lie derivative action, the latter being 
known to be the zero action.
Thus, we have shown that  $H_j(\m_\chi, K)\in\catu$.
This completes the proof.
\end{proof}

\begin{proof}[Proof of Theorem \ref{main_thm}]
It follows from the preceding paragraph that we have $K/K\mc=
H_0(\mc, K)\in\catu$,
cf. \eqref{whb}. Thus, we get a functor
$\Wh_\m: \  \hc(\U_{c'},\U_{c})\to \catu,\ K\mto K/K\mc$.
The homology vanishing of Lemma  \ref{hom_van}
implies that this functor is  {\em exact}.
The functor $\Wh^\m: \catu\to \Lmod{A_c}$ being an equivalence, cf. Proposition \ref{skr}(i),
we deduce the exactness of
the composite functor $\Wh^\m\ccirc\Wh_\m$.
The exactness statement of part (iii) of the theorem now follows
by writing $\Wh^\m_\m=\Wh^\m\ccirc\Wh_\m$.

Next,
fix an  $\ad\g$-stable
good filtration  on $K$, and write
$\Wh^\m_\m K=\Wh^\m(K/K\mc)$.
It follows, in particular, that
the induced filtration on $K/K\mc$ is $\m$-stable. Further,
by Lemma \ref{hom_van}, we get (below,
$(-)|_\sn$ stands for a restriction of a $\C[\N]$-module to 
the subvariety $\sn\sset\g^*$)
$$
\gr\Wh^\m_\m K=\gr\Wh^\m(K/K\mc)=[\gr(K/K\mc)]\,\big|_\sn=
[\gr_\kk K/(\gr_\kk K)\m_\chi]\,\big|_\sn=(\gr_\kk K) \big|_\sn.
$$
where  the second equality
is due to Proposition \ref{skr}(ii)
applied to the object $V=K/K\mc\in\catu$.
This proves part (i) of the theorem.

Observe that proving commutativity of the
diagram of part (ii) is equivalent,
thanks to  Skryabin's equivalences, cf. Proposition
\ref{skr}(i)-(ii), to showing
commutativity of the following diagram
\beq{fdiag}
\xymatrix{
\catu\ar[d]_<>(0.5){K\o_{\U_{c}}(-)}
&&&\modu{\AA_{c}}\ar[lll]_<>(0.5){Q_{c}\o_{A_{c}}(-)}
\ar[d]^<>(0.5){\Wh^\m_\m K\o_{\AA_{{c}}}(-)}\\
{{(\U_{{c'}},\mc)\mbox{-}\operatorname{mod}}}
&&&\modu{\AA_{{c'}}}.\ar[lll]_<>(0.5){Q_{c'}\o_{A_{c'}}(-)}
}
\eeq

To prove \eqref{fdiag}, write 
$\Wh^\m_\m K=\Wh^\m\ccirc\Wh_\m K$. Thus, one has a
canonical map
\beq{Phi}
Q_{c'}\o_{A_{c'}}\Wh^\m(\Wh_\m K)\stackrel{\Phi}\too \Wh_\m K=K/K\cd\m_\chi=
K\o_{\U_{c}}Q_{c}.
\eeq

Since
$\Wh_\m K\in \Lmod{(\U_{c'},\mc)},$  the map $\Phi$ is
actually an isomorphism, by
Skryabin's equivalence. 
Hence, tensoring diagram \eqref{Phi} with a left $\AA_{c}$-module $N$, 
we get a chain of isomorphisms
$$
Q_{c'}\ \o_{A_{c'}}\ \Wh^\m_\m K\ \o_{A_{c}} N\ 
=\ 
(Q_{c'}\o_{A_{c'}}\Wh^\m_\m K)\o_{A_{c}} N
\ \stackrel{_{\eqref{Phi}}}\too\
(K\o_{\U_{c}}Q_{c})\o_{A_{c}} N.
$$

 The composite  isomorphism above
 provides the  isomorphism of functors
that makes diagram  \eqref{fdiag} commute,
and  Theorem
\ref{main_thm}(ii) follows.

We now complete the proof of part (iii) of the theorem.
To this end, pick a good $\ad\g$ stable filtration on
our \whc\; bimodule $K$.
We know by part (i) that $\gr\Wh^\m_\m K=(\gr_\kk K)|_\sn.$ 
This is clearly a finitely
generated $\C[\sn\times\sn]$-module supported on the diagonal
in $\sn\times\sn$. It follows that $\Wh^\m_\m K$ is itself finitely
generated and we have
$\Wh^\m_\m K\in\WHC(A_{c'},A_{c})$.

Further, it is clear that if $\supp(\gr\Wh^\m_\m K)\sset\N^\Diamond$,
then we have $\gr\Wh^\m_\m K=(\gr_\kk K)|_\sn=0$ since
$\N^\Diamond\cap\sn=\emptyset$.
The filtration on $\Wh^\m_\m K$ is bounded below, hence,
separated. Hence, the equation $\gr\Wh^\m_\m K=0$ implies
$\Wh^\m_\m K=0$. Thus, we have shown that the functor
$\Wh^\m_\m$ factors through the quotient category
$\hc(\U_{c'},\U_{c})/\hc_{\N^\Diamond}(\U_{c'},\U_{c}).$

It remains to show that the resulting functor $\Wh^\m_\m$ is faithful.
 To prove this,  observe that
for $K\in \hc(\U_{c'},\U_{c})/\hc_{\N^\Diamond}(\U_{c'},\U_{c})$
the equation $\Wh^\m_\m K=0$ implies $K=0$, by the first equivalence
of Corollary \ref{fin}. The faithfulness of $\Wh^\m_\m$
is now a consequence of the
following general `abstract nonsense' result:
{\em Let $F: {\scr C} \to {\scr C}'$ be an exact functor between
abelian categories such that $F(M)=0$ implies $M=0$.
Then $F$ is faithful.}
\end{proof}

\subsection{Some applications} \label{rem2}
Given an algebra $B$ and a left, resp. right, $B$-module $N$,
let $\ann_BN$ denote the annihilator of $N$, a two-sided ideal in $B$.

Corollary \ref{fin} easily implies the main result of
Matumoto, \cite{Ma}. Matumoto's result says that, for any  finitely generated right $\U_c$-module
$M$, one has
\beq{mat}\Wh_\m M\neq0\quad\Longrightarrow\quad
\O\sset \var(\U_c/\ann_{\U_c} M).
\eeq

To see this, put $I:=\ann_{\U_c} M$. Since $M$ is finitely generated, one can find an integer
$n\geq1$ and an  $\U_c$-module
surjection
$(\U_c/I)^{\oplus n}\onto M$. The functor of coinvariants being right
exact, we get  a surjection 
$\Wh_\m(\U_c/I)^{\oplus n}\onto\Wh_\m M$. Hence, 
$\Wh_\m M\neq0$ implies $\Wh_\m(\U_c/I)\neq0.$
We conclude that $\Wh_\m^\m(\U_c/I)\neq0,$ by Proposition \ref{skr}(i).
Now, \eqref{mat} follows from the first equivalence
of 
Corollary \ref{fin} applied to $K=\U_c/I$.\qed
\smallskip

As a second application of our technique, we provide a new proof of a
result (Theorem \ref{prim1}
below) conjectured by Premet \cite[Conjecture 3.2]{premet2}.
In the special case of rational central characters,
the conjecture was first proved (using
characteristic $p$ methods) by Premet in \cite{premet3}, and
shortly afterwards  by Losev \cite{L} in full generality.
An alternate proof of the general case  was later obtained in \cite{premet4}.
 Our approach is totally different
from those used by Losev and Premet.

\begin{Thm} \label{prim1}
For any primitive ideal $I\sset \U_{c}$ such that
$\var(\U_{c}/I)=\overline{\O}$, there exists a simple
finite dimensional $A_{c}$-module $N$ such that one has
$\ann_{\U_c}(Q_{c}\o_{A_{c}} N)=I$.
\end{Thm}

\begin{proof} Let 
$I\sset \U_{c}$ be  a  primitive ideal
such that $\var (\U_c/I)=\overline{\O}$.
Corollary \ref{fin} implies that $\Wh^\m_\m  (\U_c/I)$ is
a finite dimensional $(A_{c},A_{c})$-bimodule.

Let $N$ be a simple left $A_{c}$-submodule of 
 $\Wh^\m_\m  (\U_c/I)$, the latter being viewed
as a {\em left} $A_{c}$-module. 
By Skryabin's theorem
\ref{skr}, we have a diagram
of $\U_c$-modules
$$Q_{c}\o_{A_{c}} N\into
Q_{c}\o_{A_{c}}\Wh^\m_\m  (\U_c/I)=\Wh_\m(\U_c/I)=\U_c/(I+\U_c\mc)
\,\twoheadleftarrow\, \U_c/I.
$$

Let $J=\ann_{\U_c}(Q_{c}\o_{A_{c}} N).$
From the above diagram, we get
$J\supset \ann_{\U_c}\U_c/(I+\U_c\mc)\supset I$.
Moreover, $\var(\U_c/I)$ is an $\Ad G$-stable subset of $\N$,
and we have
$\{\chi\}\sset\var (Q_{c}\o_{A_{c}}N)\sset \var (\U_c/J)$.
Hence, $\O\sset \var(\U_c/J)$ and we get
$\dim \var (\U_c/I)=\dim\O\leq \dim\var(\U_c/J)$.
Since, $I\sset J$,
we conclude that $\dim \var (\U_c/I)=\dim\var(\U_c/J)$.

Now,  $N$ being a simple finite dimensional 
left $A_{c}$-module, we deduce using Skryabin's equivalence that 
$Q_{c}\o_{A_{c}} N$ is a simple $\U_c$-module. Thus,
 $J$  is a primitive ideal in $\U_c$.
The equation $\dim \var (\U_c/I)=\dim\var(\U_c/J)$ combined with the inclusion
$I\sset J$ forces
$I=J,$ due to
a result by Borho-Kraft \cite{BoK}, Korollar 3.6.
\end{proof}

\begin{rem} Set $n:=\dim\sn$. 
The assignment $N\mto \Ext^n_{\Lmod{A_c}}(N, A_c)$ is an
{\em exact} functor on the category
of finite dimensional left $A_c$-modules, thanks to Corollary \ref{gorenst}.
That functor gives a contravariant duality
$\WHC_{\{\chi\}}(A_c,A_c)\iso\WHC_{\{\chi\}}(A_c,A_c)^\op.$
\end{rem}

\section{$\scr D$-modules}
\subsection{Whittaker $\scr D$-modules}\label{Dmod}
Let $\h$ denote the abstract Cartan algebra for the Lie aldgebra $\g$,
and let $\xx\sset\h^*$ be the semigroup of dominant integral weights.

For any integral weight $\la\in\h^*$ one has a $G$-equivariant line
bundle $\oo(\la)$ on $\B$. For any $\nu\in\H^*$,
let $\dd_\nu$ denote the sheaf of $\nu$-{\em twisted}
algebraic differential operators on $\B$, see \cite{BB}.
In the case where $\nu$ is integral, we have $\dd_\nu=\dd(\oo(\nu))$,
is the sheaf of differential operators acting in the sections of the
 line bundle $\oo(\nu)$.
 There is a canonical algebra isomorphism $\U_\nu\cong\Gamma(\B,\dd_\nu)
$, see \cite{BB}.
 
Let  $\mathcal{V}$ be  a coherent 
$\dd_\nu$-module on $\mathcal{B}$
that has, viewed 
as  a quasi-coherent $\oo_\B$-module, an additional $M$-equivariant structure.
Given an element  $x\in \m$, we write 
$x_{_\dd}$ for the action on $\mathcal{V}$
of the vector field corresponding
to $x$ via the $\dd_\nu$-module structure,
and $x_{_M}$ for the action on $\mathcal{V}$
obtained by differentiating the $M$-action
arising from the equivariant structure.

We say that   an
 $M$-equivariant
$\dd_\nu$-module  $\mathcal{V}$ 
is an $(\m,\chi)$-Whittaker $\dd_\nu$-module 
 if, for any $x\in \m$ and $v\in \mathcal{V}$, we have
$(x_{_\dd}-x_{_M})v=\chi(x)\cdot v$.
Write $\cat$ for the abelian category of  $(\m,\chi)$-Whittaker 
coherent  $\dd_\nu$-modules on $\B$.

We put  $\qq_\nu:=\dd_\nu/\dd_\nu\cdot\mc$.
This 
$\dd_\nu$-module  is clearly
an object of $\cat$.
For any $\dd_\nu$-module   $\mathcal{V}$ from the definitions one finds
$\Hom_{\dd_\nu}(\qq_\nu,  \mathcal{V})=\Wh^\m(\Ga(\B,   \mathcal{V}))$.
This space has a right $A_\nu$-module structure.
In particular, taking $\mathcal{V}=\qq_\nu$,
one obtains an algebra homomorphism
$\AA_\nu^{\op}\to\Hom_{\dd_\nu}(\qq_\nu,
\qq_\nu)$
that makes $\qq_\nu$ a $(\dd_\nu, \AA_\nu)$-bimodule.

\begin{rem} The natural $\dd$-module projection
$\dd_\nu\onto\qq_\nu$ induces
an $\U_\nu$-module
map $Q_\nu=\U_\nu/\U_\nu\mc\to\Ga(\B, \qq_\nu)$.
The latter map turns out to be an isomorphism,
according to a special case ($\la=0$) of
Corollary \ref{Q}(ii). If $\nu$ is a dominant
weight then the same isomorphism follows
from the Beilinson-Bernstein theorem \cite{BB}.
In any case, one concludes that 
the canonical algebra map $\AA_\nu^{\op}\to\Hom_{\dd_\nu}(\qq_\nu,
\qq_\nu)$
is an isomorphism as well.
\end{rem}

For any regular and dominant $\nu\in\h^*$, the 
localization theorem of Beilinson and Bernstein \cite{BB} yields a category
equivalence $\Gamma(\B,-): \cat\iso{(\U_\nu,\mc)\mbox{-}\operatorname{mod}}$.
Combining the Beilinson-Bernstein and
the Skryabin equivalences, 
one obtains, see \cite[Proposition 6.2]{GG}, the 
following result

\begin{Prop}\label{BBprop} For  any regular and dominant $\nu\in\h^*$,
the following functors provide mutually inverse  equivalences,
 of abelian categories,
$$
\xymatrix{
\cat\
\ar@<0.5ex>[rrrr]^<>(0.5){\Hom_{\dd_\nu}(\qq_\nu,-)}&&&&
\ \modu{\AA_\nu}\ar@<0.5ex>[llll]^<>(0.5){\qq_\nu\o_{\AA_\nu}(-)}
}\eqno\Box.$$
\end{Prop}

For any Borel subalgebra $\b$, let $\C_\b$ denote the
corresponding
1-dimensional $\oo_\B$-module, a skyscraper sheaf at the
point $\{\b\}\sset\B$.
Recall diagram \eqref{diagram}. The following result is standard
\begin{Lem}\label{vanish}
For any $\b\in\B\sminus \ps(\si)$ and any $\cv\in\cat,$ we have
$\,\C_\b\stackrel{L}\o_{\oo_\B}\cv=0$.
\end{Lem}
\begin{proof} Fix $\b\in\B$ and write
$\n:=[\b,\b]$ for the nil-radical  of $\b$.
Given a $\dd_\nu$-module $\cv$, let $V:=\Gamma(\B,\cv)$
be the corresponding $\U_\nu$-module. The
Beilinson-Bernstein theory yields, in particular,
a canonical isomorphism
$$\Tor_\idot^{\oo_\B}(\C_\b,\,\cv)=H_\idot(\n,V)_o,$$
where $H_\idot(\n,-)$ denotes the Lie algebra homology functor,
and the subscript `$o$' stands for a certain
particular weight space of the natural Cartan subalgebra action
on homology.

Assume now that $\cv\in\cat$. Then,
$V\in\catu$. Therefore, for any element
$x\in \m\cap\n$, the natural action of $x$
on $H_\idot(\n,V)$ is such that the operator
$x-\chi(x)$ is nilpotent. On the other hand,
the action of any Lie algebra on its homology
is trivial, hence, the element
$x$ induces the zero operator on $H_\idot(\n,V)$.
Thus, we must have $\chi(x)=0$. Proposition \ref{twisted}(i) completes
the proof.
\end{proof}

\subsection{Translation bimodules} \label{trans}
Let $\oo_{T^*\B}(\la)$ be the pull-back of the line bundle $\oo(\la)$,
on $\B$, via the cotangent bundle projection
$T^*\B\onto\B$ and let  $\oo_{\ss}(\la)$ denote
the restriction of  the line bundle $\oo_{T^*\B}(\la)$
to the Slodowy variety $\ss\sset T^*\B$.
 The sheaf  $\oo_{\ss}(\la)$ is equivariant with respect
to the $\bullet$-action on $\ss$. This gives the
space $\Ga(\ss,\ \oo_{\ss}(\la))$ a $\Z$-grading which is
 bounded from below since $\Ga(\ss,\ \oo_{\ss}(\la))$
is a finitely generated $\C[\sn]$-module.

For any $\nu\in\h^*$ and an integral weight $\la$, we put
$\dd_{\nu}^{\nu+\la}:=\oo(\la)\o_{\oo_\B}\dd_\nu$,
resp. $\qq_\nu^{\nu+\la}:=\oo(\la)\o_{\oo_\B}\qq_\nu$.
Let $U_{\nu}^{\nu+\la}:=\Ga(\B,\ \dd_{\nu}^{\nu+\la})$,
resp. $Q_\nu^{\nu+\la}:=\Ga(\B,\ \qq_\nu^{\nu+\la})$.
Further, we define 
$$\AA_\nu^{\nu+\la}=(Q_\nu^{\nu+\la})^{\ad\m}=
\Wh^\m(Q_\nu^{\nu+\la})=\Wh^\m\big(\Ga(\B,\ \oo(\la)\o_{\oo_\B}\qq_\nu)\big).$$
The space $\AA_\nu^{\nu+\la}$ comes equipped with a natural $(\AA_{\la+\nu},
\AA_\nu)$-bimodule
structure.

The standard filtration on
$\dd_\nu$ by the order of differential operator
gives rise to a tensor product filtration
on $\oo(\la)\o_{\oo_\B}\dd_\nu$, where the factor
$\oo(\la)$ is assigned filtration degree 0.
 This induces
a natural ascending filtration on the vector space $Q_\nu^{\nu+\la}$,
resp. $\AA_\nu^{\nu+\la}$.
Further, the $\bullet$-action on $\B$ gives 
a $\Z$-grading 
on the above vector spaces. Hence, formula
\eqref{FK} provides the vector space $Q_\nu^{\nu+\la}$,
resp. $\AA_\nu^{\nu+\la}$,  with
a natural Kazhdan filtration $F_\kk Q_\nu^{\nu+\la}$,
resp. $F_\kk\AA_\nu^{\nu+\la}$.

Main propeties of the bimodules $\AA_\nu^{\nu+\la}$ are summarized in
the following result.

\begin{Prop}\label{equiv}  Let $\nu$  be dominant   regular
weights.
Then, for any  $\la,\mu\in\xx$,
we have 

\vi $\AA_\nu^{\nu+\la}$  is  finitely generated and projective as a
left $\AA_{\nu+\la}$-module, as well as a right $\AA_\nu$-module;

\vii The natural pairing
$\AA^{\nu+\la+\mu}_{\nu+\la}\ \o_{\AA_{\nu+\la}}\
\AA_\nu^{\nu+\la}\to \AA_\nu^{\nu+\la+\mu}$
is an isomorphism;

\viii If $\la$ is  sufficiently dominant, then there is  a canonical graded space isomorphism
\beq{grH}
\gr_\kk \AA_\nu^{\nu+\la}=
\Ga(\ss,\ \oo_{\ss}(\la)),
\eeq which is compatible with the pairings in (ii).

\iv The following {\em translation functor}  is an equivalence
$$\T^{\nu+\la}_\nu:\
\modu{\AA_\nu}\to\modu{\AA_{\nu+\la}},\quad
M\mto \AA^{\nu+\la}_\nu\o_{\AA_\nu} M.$$
\end{Prop}

\begin{rem}
Applying $\gr(-)$
to the pairing of part (ii) of the Proposition,
one obtaines a pairing
$$\gr_\kk \AA^{\nu+\la+\mu}_{\nu+\la}\ \o_{\gr_\kk \AA_{\nu+\la}}\
\gr_\kk \AA_\nu^{\nu+\la}\to \gr_\kk \AA_\nu^{\nu+\la+\mu}.$$
The last statement of Proposition \ref{equiv}(iii) means that the latter
pairing 
gets identified, via the the  isomorphisms in \eqref{grH},
with the natural pairing 
$$\Ga(\ss, \oo_\ss(\la))\,\o\, \Ga(\ss, \oo_\ss(\mu))\to \Ga(\ss, \oo_\ss(\la+\mu)),$$
induced by the sheaf morphism $\oo_\ss(\la)\o\oo_\ss(\mu)\to
\oo_\ss(\la+\mu).$
\end{rem}

We  begin the proof of Proposition \ref{equiv} with
the following result that relates `geometric' and
`algebraic' translation functors.

\begin{Lem}\label{trans2} For any  dominant   regular
weight $\nu$ and any $\la\in\xx$, the following diagram commutes
$$
\xymatrix{
{{(\dd_\nu,\mc)\mbox{-}\operatorname{mod}}}
\ar[d]_<>(0.5){\oo(\la)\otimes_{\oo_\B}(-)}
\ar[rrr]^<>(0.5){\Wh^\m\Ga(\B,-)}
&&&\modu{\AA_\nu}\ar[d]^<>(0.5){\T^{\nu+\la}_\nu}\\
{{(\dd_{\nu+\la},\mc)\mbox{-}\operatorname{mod}}}
\ar[rrr]^<>(0.5){\Wh^\m\Ga(\B,-)}
&&&\modu{\AA_{\nu+\la}}
}
$$
\end{Lem}
\begin{proof} First of all, we claim that the projection
$\dd^{\nu+\la}_\nu\onto\qq^{\nu+\la}_\nu$ induces an isomorphism
\beq{Qshift}
A^{\nu+\la}_\nu=\Wh^\m_\m\Ga(\B, \dd^{\nu+\la}_\nu).
\eeq

To prove this, we are going to use Corollary \ref{Q} from
section \ref{Dbimod} below, which is independent of the intervening
material. 
We have
$$A^{\nu+\la}_\nu=\Wh^\m Q^{\nu+\la}_\nu
=\Wh^\m\Ga(\B, \qq^{\nu+\la}_\nu)=
\Wh^\m\big[\Ga(\B, \dd^{\nu+\la}_\nu)/\Ga(\B,
\dd^{\nu+\la}_\nu)\mc\big],
$$
where the last equality holds thanks to part (ii) of  Corollary
\ref{Q}. The rightmost term in the above
chain of equalities is nothing but
$\Wh^\m_\m\Ga(\B, \dd^{\nu+\la}_\nu),$
and formula \eqref{Qshift} is proved.

Now, let $\cv\in {(\dd_\nu,\mc)\mbox{-}\operatorname{mod}}$.
One clearly has
$\oo(\la)\otimes_{\oo_\B}\cv=\dd^{\nu+\la}_\nu\o_{\dd_\nu}\cv$.
Therefore, applying the Beilinson-Bernstein equivalence,
we obtain 
\beq{abc}\Ga\big(\B,\ \oo(\la)\otimes_{\oo_\B}\cv\big)=
\Ga(\B, \dd^{\nu+\la}_\nu)\o_{\U_\nu}\Ga(\B, \cv).\eeq

We put $K:=\Ga(\B, \dd^{\nu+\la}_\nu)$ so that
\eqref{Qshift} reads
$A^{\nu+\la}_\nu=\Wh^\m_\m K$. By Theorem
\ref{main_thm}(ii),  we get
$$\Wh^\m\big(K\o_{\U_\nu}\Ga(\B, \cv)\big)=
\Wh^\m_\m K\o_{A_\nu}\Wh^\m\Ga(\B, \cv)=
A^{\nu+\la}_\nu\o_{A_\nu}\Wh^\m\Ga(\B, \cv).
$$

The leftmost term  in the last formula equals
$\Wh^\m\Ga\big(\B,\ \oo(\la)\otimes_{\oo_\B}\cv\big),$
by \eqref{abc}. Thus, we have proved that
$\Wh^\m\Ga\big(\B,\ \oo(\la)\otimes_{\oo_\B}\cv\big)=
\T^{\nu+\la}_\nu\big(\Ga(\B, \cv)\big),$ and the
lemma follows.
\end{proof}

\begin{proof}[Proof of Proposition \ref{equiv}]
For dominant and regular $\nu,\la$, 
the horizontal functors in each row of the diagram
of Lemma \ref{trans2} are
the equivalences of Proposition \ref{BBprop}.
Further,  the functor $\oo(\la)\o_{\oo_\B}(-)$
given by the vertical arrow on the left  of the diagram is
 clearly an
equivalence. We conclude that the functor
$\T^{\nu+\la}_\nu$
given by the vertical arrow on the right
 of the diagram is  an
equivalence as well, and (iv) is proved.

Note that the equivalence of part (iv)
yields, in particular, an algebra isomorphism
\beq{end}
\End_{\AA_{\nu+\la}}(\AA^{\nu+\la}_\nu)=
\End_{\AA_{\nu+\la}}(\T^{\nu+\la}_\nu(\AA_\nu))=
\End_{\AA_\nu}(\AA_\nu)=\AA_\nu^{\op}.
\eeq

Thus, the $(\AA_{\nu+\la}, \AA_\nu)$-bimodule $\AA^{\nu+\la}_\nu$
fits the standard Morita context. Hence,  the general Morita
theory implies all the statements from part (i) of the proposition.

We now prove  (ii). To this end, we observe that the composite
functor
$\T^{\nu+\la+\mu}_{\nu+\la}\ccirc\T_\nu^{\nu+\la}$ is
clearly given by tensoring with $\AA^{\nu+\la+\mu}_{\nu+\la}\
\o_{\AA_{\nu+\la}} \AA_\nu^{\nu+\la}$,
an
$(\AA_{\nu+\la+\mu}, \AA_\nu)$-bimodule.
Therefore, the canonical pairing in (ii) induces
a morphism of functors
\beq{psi}
\T^{\nu+\la+\mu}_{\nu+\la}\ccirc\T_\nu^{\nu+\la}\to
\T_\nu^{\nu+\la+\mu}.
\eeq

  We may transport the latter
morphism via the equivalences provided
by Proposition \ref{BBprop}. In this way, we get a morphism
$\oo(\mu)\o_{\oo_\B}(-)\,\ccirc\,
\oo(\la)\o_{\oo_\B}(-)\too\oo(\mu+\la)\o_{\oo_\B}(-),
$
of functors $\cat\to
{{(\dd_{\nu+\la+\mu},\mc)\mbox{-}\operatorname{mod}}}$.
It is clear from commutativity of diagram \eqref{comm_phi}
that the latter  morphism of functors is 
the one induced by the canonical morphism of sheaves
$\psi: \oo(\mu)\o_{\oo_\B}\oo(\la)\to\oo(\mu+\la)$.

Now, the morphism of sheaves $\psi$ is clearly an isomorphism.
It follows that the associated morphism of functors is
 an isomorphism as well. Thus, we conclude that the
 morphism in \eqref{psi} is  an isomorphism.
For the corresponding bimodules, this implies that the pairing 
in (ii) must be an isomorphism.

The proof of part (iii) of the proposition will be given in \S\ref{Dbimod}.
\end{proof}

\subsection{Characteristic varieties}\label{charvar}
We are going to define a  {\em Kazhdan filtration} on $\dd_\nu$.
To this end,
view $\B$ as a $\gm$-variety via the action $\gm\ni t:\
\b\mto \Ad\gamma_t(\b).$
Any $\gm$-orbit in an arbitrary quasi-projective $\gm$-variety
is known to be contained in an 
{\em affine} Zariski-open $\gm$-stable subset. Thus, we
may view $\dd_\nu$ as a sheaf in the topology formed by
 Zariski-open $\gm$-stable subsets of $\B$.

For any   Zariski-open subset
$U\sset \B$, the order filtration on differential operators 
gives a filtration on the vector space  $\Ga(U,\   \dd_\nu)$.
If, in addition, $U$ is  $\gm$-stable, then the
 $\gm$-action gives a weight decomposition
$\Ga(U,\   \dd_\nu)=\bplus_{i\in\Z}\ \Ga(U,\   \dd_\nu)\langle
i\rangle$.
We are therefore in a position to define an associated Kazhdan 
filtration on
$\Ga(U,\   \dd_\nu)$ by formula \eqref{FK}.
 For the  associated
graded sheaf, one has a canonical isomorphism 
$\gr_\kk\dd_\nu=p_\idot\oo_{T^*\B}$.

Let $\kk_\idot \cv$ be a good $M$-stable Kazhdan filtration on
a $\dd$-module $\cv\in\cat$.
Write ${\wt\gr}_\kk\cv$ for the $M$-equivariant coherent sheaf on
$T^*\B$
such that $p_\idot{\wt\gr}_\kk\cv=\gr_\kk\cv.$

For any weight $\la$, the filtration on
$\cv$ induces one on $\oo(\la)\o_{\oo_\B}\cv$,
hence, also on the vector space
$\Ga(\B,\ \oo(\la)\o_{\oo_\B}\cv)$ and on
$\Wh^\m\Ga(\B,\ \oo(\la)\o_{\oo_\B}\cv)$, by
restriction.

\begin{Lem}\label{filt} 
In the above setting, for all sufficiently dominant $\la\in\xx$,
there is a canonical isomorphism
$$\Ga\big(\ss, \ \oo_\ss(\la)\o_{\oo_\ss}(\tgr_\kk\cv|_\ss)\big)\iso
\gr_\kk\Wh^\m\big(\Ga(\B,\ \oo(\la)\o_{\oo_\B}\cv)\big).$$
\end{Lem}

\begin{proof} 
Since $\cv\in\cat$, any good Kazhdan filtration $\kk_\idot\cv$
on $\cv$
is bounded below and
we have $\supp\tgr_\kk\cv\sset\si$. 
For any integral $\la$, put $\cv(\la)=\oo(\la)\o_{\oo_\B}\cv$.
It is clear that
 $\kk_\idot\cv(\la)=\oo(\la)\o_{\oo_\B}\kk_\idot\cv.$
Thus, $\tgr_\kk\cv(\la)=\oo_\si(\la)\o_{\oo_\si}\tgr_\kk\cv$
is a coherent sheaf supported on $\si$. 

We  apply the functor $\Wh^\m\Ga(\B,-)$ to
$\cv(\la)$, a filtered sheaf. Thus, 
there is a standard
convergent
spectral sequence involving $R^\hdot\Wh^\m\Ga(\B,-)$,
the right derived functors of the composite functor
$\Wh^\m\Ga(\B,-)$. The spectral sequence reads
\beq{Ewh}
E_1=R^\hdot\Wh^\m\Ga(\B,\ \gr_\kk\cv(\la))\big)\en
\Rightarrow\en
\gr_\kk\big(R^\hdot\Wh^\m\Ga(\B,\ \cv(\la))\big),
\eeq
where $R^\hdot\Wh^\m\Ga(\B,-)$ stand for the derived functors of
 $\Wh^\m\Ga(\B,-)$, a left exact functor.

We claim that, for all  $\la$ dominant enough,
one has $R^i\Wh^\m\Ga(\B,\ \gr_\kk\cv(\la))\big)$
for any $i>0$. 
That would yield the collapse of the above spectral sequence
which would give, in turn, the required canonical isomorphism
$$\Wh^\m\Ga(\B,\ \gr_\kk\cv(\la))\iso
\gr_\kk\Wh^\m\Ga(\B,\ \cv(\la)).
$$

To complete the proof we must check the above claim
that  $R^i\Wh^\m\Ga(\B,\ \gr_\kk\cv(\la))=0$
for any $i>0$.
 To this end,  write
$q:\ \si\onto\si/M=\ss$ for the  projection,
cf. \eqref{isom_si}. Observe further that
 applying the
functor $\Wh^\m$ to the
vector space $\Ga(\si,\ \tgr\cv(\la))$ amounts to taking
$M$-invariants. Thus, writing $\invar^M$ for the functor
of $M$-invariants, we get
$$\Wh^\m\Ga(\B,\ \gr_\kk\cv(\la))=\invar^M\Ga(\si,\ \tgr\cv(\la))=
\invar^M\Ga(\ss,\ q_\idot\tgr\cv(\la))
=\Ga(\ss,\ \invar^M\!q_\idot\tgr\cv(\la)).
$$

The $M$-action on $\si$ being free, one
can repeat
the argument in the proof of \cite{GG}, formula (6.1) and
Proposition 5.2, to show that the functor
$\invar^M\!q_\idot$ is exact (and is isomorphic to the
functor of restriction to the closed submanifold $\ss\sset\si$).
Further, using the  Springer resolution $\pi: \ss\to\sn,$
we can write
$\Ga(\ss,-)=\Ga(\sn, \pi_\idot(-))$.
Therefore, for any $i\geq 0,$ we have  isomorphisms
of derived functors
$$R^i\Wh^\m\Ga(\B,-)\cong R^i\Ga\big(\ss,\ \invar^M\!q_\idot(-)\big)
=\Ga\big(\sn,\ R^i\pi_\idot(\invar^M\!q_\idot(-))\big),
$$
where in the last isomorphism we have used that
$\Ga(\sn,-)$ is an exact functor since $\sn$ is affine.

Recall next that, for a dominant regular weight
$\la$ 
the sheaf $\oo_\si(\la)$ is relatively ample with
respect to the Springer resolution $\pi$.
Therefore, for $\la$ dominant enough, one has
$R^i\pi_\idot(\invar^M\!q_\idot\tgr_\kk\cv(\la))=0$,
for any $i>0$. Hence, for such $\la$, and $i>0$, we obtain
$R^i\Wh^\m\Ga(\B,\ \gr_\kk\cv(\la))=0$, and we are done.
\end{proof}

\subsection{Harish-Chandra $\dd$-modules}\label{Dbimod}
For any pair $\mu,\nu\in\h^*,$ we put
 $\dd_{\mu,\nu}:=\dd_\mu\boxtimes\dd_\nu$, a sheaf of
twisted differential operators on $\BB$.
Let  $G^{sc}$ denote 
a simply-connected cover of
the semisimple group $G$.
The group  $G^{sc}$ acts diagonally 
on $\BB$. We may consider the restriction of this action 
to the 1-parameter subgroup  $\C^\times$ 
and corresponding Kazhdan filtrations on the sheaf $\dd_{\mu,\nu}$ as
well as on
 $\dd_{\mu,\nu}$-modules.

Recall that a  $\dd_{\mu,\nu}$-module
is called $G$-monodromic if it is
$G^{sc}$-equivariant (with respect to the diagonal 
action on $\BB$).
Let   $U\sset\BB$ a Zariski open  subset stable under the
$\C^\times$-diagonal action.
Then,
for any  $G$-monodromic  $\dd_{\mu,\nu}$-module $\cv$,
the induced $\C^\times$-action on
$\Ga(U, \cv)$  is locally finite. Hence, any $G$-stable filtration on
$\cv$ gives an associated Kazhdan filtration \eqref{FK}.

We define a set 
$\ZZ:=\{u\times v\in T^*\B\times T^*\B\mid \pi(u)=-\pi(v)\}$,
where $\pi$ is the Springer resolution  \eqref{TB}.
We will view $\ZZ$ as a closed {\em reduced} subscheme in
$T^*\B\times T^*\B=T^*(\B\times\B)$,
called {\em Steinberg variety}. The assignment
$u\times v\mto \pi(u)$ gives a proper morphism
$\pi:\ \ZZ\to\N$.
For any $G$-monodromic $\dd_{\mu,\nu}$-module ${\mathcal V}$, one has
$\var{\mathcal V}\sset\ZZ$.

Let   ${\mathcal V}$
be a coherent $\dd_{\mu,\nu}$-module
and let $pr:\ \BB\to\B$ denote the first projection. We
abuse the notation and write ${\mathcal V}/{\mathcal V}\mc:=
pr_\idot[{\mathcal V}/(1\boxtimes\mc){\mathcal V}]$ for a 
sheaf-theoretic direct image  of the sheaf of coinvariants
with respect to the action of the Lie algebra
$1\boxtimes\mc\sset\dd_\mu\boxtimes\dd_\nu$.
A filtration on ${\mathcal V}$ induces
one on ${\mathcal V}/{\mathcal V}\mc$.

\begin{Prop}\label{dmc} Let ${\mathcal V}$ be a $G$-monodromic
$\dd_{\mu,\nu}$-module equipped with a $G$-stable good
filtration. Then, the induced filtration on 
the $\dd_\mu$-module ${\mathcal V}/{\mathcal
V}\mc$
is good and we have
${\mathcal V}/{\mathcal V}\mc\in\cat$.
Furthermore, the corresponding Kazhdan filtration on
${\mathcal V}/{\mathcal V}\mc$
is good as well.

In addition,  we have  $ H_j(\m_\chi,\gr_\kk {\mathcal V})=0$
and $H_j(\m_\chi, {\mathcal V})=0,$ for any $j>0$; moreover, the canonical map
$\gr_\kk{\mathcal V}/(\gr_\kk{\mathcal V})\mc\to
\gr_\kk({\mathcal V}/{\mathcal V}\mc)$ is an isomorphism.
\end{Prop}

\begin{proof} Given a  $G$-stable good
filtration $F_\idot\cv$, write ${\wt\gr}_F\cv$
for the  coherent $\oo_{T^*\B\times T^*\B}$-module such
that $p_\idot{\wt\gr}_F\cv=\gr_F\cv$.
Then, we have that $\tgr_F\cv$ is a
$G^{sc}$-equivariant
coherent sheaf on $\ZZ$. The composite
$\wt\pr:\ \ZZ\into T^*\B\times T^*\B\to T^*\B$,
of the closed imbedding and the first projection,
is a proper morphism. Hence, the push-forward $\wt\pr_\idot({\wt\gr}_F\cv)$
is a coherent sheaf on $T^*\B$. It follows
that the induced filtration on  ${\mathcal V}/{\mathcal V}\mc$,
viewed as a left  $\dd_\mu$-module,
is good and that  ${\mathcal V}/{\mathcal V}\mc$
is a coherent  $\dd_\mu$-module.

Now, applying Corollary \ref{flat}(ii)
to the morphism $\pi:\ \ZZ\to\N$ we deduce that
the stalk of the sheaf $\tgr_F\cv$
at any point of $\pi\inv(\chi+\m^\perp)$ is a flat
$(\varpi\ccirc\pi)^\hdot\oo_\chi$-module. 
At this point, the proof of Lemma \ref{hom_van}
(cf. the proof of Theorem \ref{main_thm}
given in \S\ref{pf_sec}) goes through verbatim in our present
situation. 
\end{proof}

The Kazhdan filtration on $\dd_\nu$
induces a good Kazhdan filtration on each of the objects
$\dd_{\nu}^{\nu+\la},
\ \qq_\nu,$ and also on  $\qq_\nu^{\nu+\la}.$
One may consider $\dd_\nu$ as a $\dd_\nu$-bimodule,
resp. $\dd_{\nu}^{\nu+\la}$ as a $(\dd_{\nu+\la},\dd_{\nu})$-bimodule.
Such a bimodule may be viewed as a  $G$-monodromic left
$\dd_{\nu,\mu}$-module, resp.
$\dd_{\nu+\la,\mu}$-module, supported on the diagonal
$\B\sset\BB$. Here, the weight $\mu$ is given by the formula
$\mu=-w_0(\la)+2\rho$, where $w_0$ stands for the longest element in the
Weyl group.

Recall that the group $\C^\times$ acts on $\ss$ via the $\bullet$-action.

\begin{Cor}\label{Q} For an $\nu,\la\in\h^*$, we have

\vi 
 $\qq_\nu^{\nu+\la}\in \cat$ is a Cohen-Macaulay left $\dd_{\nu+\la}$-module
and there is an $M\times \C^\times$-equivariant isomorphism
$\gr_\kk\qq_\nu^{\nu+\la}=p_\idot\oo_\si(\la)$.

\vii If $\la\in\xx$, then we have $H^j(\B, \qq_\nu^{\nu+\la})=0$
for any $j>0$; moreover, the projection
$\dd_\nu^{\nu+\la}\onto\qq_\nu^{\nu+\la}$
induces an isomorphism
$$\Gamma(\B,\dd_\nu^{\nu+\la})/\Gamma(\B,\dd_\nu^{\nu+\la})\mc
\ \iso\
\Gamma(\B,\qq_\nu^{\nu+\la})=Q_\nu^{\nu+\la}.
$$
\end{Cor}

Recall that the Cohen-Macaulay property claimed in part (i)
of the corollay says that the sheaves
${\scr E}xt^\hdot_{\dd_{\nu+\la}}(\qq_\nu^{\nu+\la},\dd_{\nu+\la})$
vanish in all degrees but one.

\begin{proof} We  view $\cv:=\dd_\nu^{\nu+\la}$
as a $\dd_{\nu+\la,\mu}$-module as explained above.
Applying Proposition \ref{dmc}, we deduce
the isomorphisms $\gr_\kk\qq_\nu^{\nu+\la}\cong
\gr_\kk \dd_\nu^{\nu+\la}/(\gr_\kk \dd_\nu^{\nu+\la})\mc\cong
p_\idot\oo_\si(\la)$.
In addition, we deduce that the
Chevalley-Eilenberg complex $\dd_\nu^{\nu+\la}\otimes\wedge^\hdot\mc$
is acyclic in positive degrees.
Thus, forgetting the right action, we see that the
Chevalley-Eilenberg complex provides a resolution
of $\qq_\nu^{\nu+\la}$ by locally free 
 left $\dd_{\nu+\la}$-modules.

To prove the Cohen-Macaulay property, we
use the Chevalley-Eilenberg complex as  a resolution
of $\qq_\nu^{\nu+\la}$.
Thus, the sheaves ${\scr
E}xt^\hdot_{\dd_{\nu+\la}}(\qq_\nu^{\nu+\la},\dd_{\nu+\la})$
may be obtained as  cohomology groups of the complex
\beq{ext}
{\scr H}om_{\dd_{\nu+\la}}(\dd_\nu^{\nu+\la}\otimes\wedge^\hdot\mc,
\,
\dd_{\nu+\la})=
{\scr H}om_{\dd_{\nu+\la}}(\oo(\la)\o\dd_\nu,\,
\dd_{\nu+\la})\otimes(\wedge^\hdot\m_\chi)^*.
\eeq

Recall that the functor $\oo(\la)\o(-)$ is an equivalence
and  $\dd_{\nu+\la}=\oo(\la)\o\dd_\nu\o
\oo(-\la)$. Hence, we get
an isomorphism
$${\scr H}om_{\dd_{\nu+\la}}(\oo(\la)\o\dd_\nu,\,
\dd_{\nu+\la})={\scr H}om_{\dd_\nu}(\dd_\nu,\,\dd_\nu\o
\oo(-\la))=
\oo(-\la)\o(\dd_\nu)^{op}.
$$

Put $m:=\dim\mc$ and recall that one has a
canonical isomorphism 
$(\dd_\nu)^{op}=\dd_\mu$ where
$\mu=-w_0(\nu)+2\rho$. 
For any $j\geq 0,$
we deduce
$$
{\scr H}om_{\dd_{\nu+\la}}(\dd_\nu^{\nu+\la}\otimes\wedge^j\mc,
\,
\dd_{\nu+\la})=
\oo(-\la)\o(\dd_\nu)^{op}\o\wedge^j\m_\chi^*
=[\dd_\mu^{\mu-\la}\o\wedge^{m-j}\m_\chi]\o
\wedge^m\mc^*.
$$

The Lie algebra $\mc$ being nilpotent, one has
an isomorphism $\wedge^m\m_\chi^*\cong\C$, of 
$\mc$-modules.
We see that the complex
on the right of \eqref{ext} may be
identified, up to degree shift, with the Chevalley-Eilenberg complex
$\dd_\mu^{\mu-\la}\o\wedge^\hdot\mc$. But
the latter complex
is a resolution of
the $\dd_{\mu-\la}$-module $\qq_{\mu-\la}$,
by the first paragraph of the proof.
Therefore, the complex \eqref{ext}
has a single nonvanishing cohomology
group which is isomorphic to $\qq_{\mu-\la}$.
This completes the proof of part (i).

To prove part (ii) we recall that,
for $\la\in\xx$ and any $j>0$, one has  the cohomology vanishing
 $H^j(T^*\B,\ \oo_{T^*\B}(\la))=0$, cf. \cite{Br}.
Using that $\tgr\dd_\nu^{\nu+\la}=\oo_{T^*\B}(\la)$,
by a standard spectral sequence argument, we deduce
 $H^j(\B,\dd_\nu^{\nu+\la})=0$ for all $j>0$.
Thus, the Chevalley-Eilenberg complex $\dd_\nu^{\nu+\la}\otimes\wedge^\hdot\mc$
yields a $\Gamma$-acyclic resolution of $\qq_\nu^{\nu+\la}$,
the latter being viewed as a sheaf  on $\B$.
We conclude that the sheaf cohomology groups $H^j(\B,\qq_\nu^{\nu+\la})$ may be
computed as the cohomology groups of the complex
\beq{complex}
\ldots\to \Gamma(\B,\,\dd_\nu^{\nu+\la}\otimes\wedge^3\mc)\to \Gamma(\B,\,\dd_\nu^{\nu+\la}\otimes\wedge^2\mc)\to
 \Gamma(\B,\,\dd_\nu^{\nu+\la}\otimes\mc)\to\Gamma(\B,\,\dd_\nu^{\nu+\la}).
\eeq

Next we apply Lemma \ref{hom_van} to 
$K:=\Gamma(\B,\,\dd_\nu^{\nu+\la})$, a Harish-Chandra
$(\U_{\nu+\la},\U_\mu)$-bimodule. We conclude that
the complex \eqref{complex} is acyclic in positive
degrees and, in degree 0, we have
$H^0(\B,\qq_\nu^{\nu+\la})=K/K\mc$.
This proves part (ii) of the corollary.
\end{proof}

\begin{proof}[Proof of Proposition \ref{equiv}(iii)] For any weights
$\nu,\la$, we 
have 
\beq{ttt1}\qq_\nu^{\nu+\la}=\oo(\la)\o_{\oo_\B} \dd_\nu/
(\oo(\la)\o_{\oo_\B} \dd_\nu)\mc=
\oo(\la)\o_{\oo_\B}(\dd_\nu/\dd_\nu\mc).
\eeq

We apply Lemma \ref{filt}
to the $\dd_\nu$-module $\cv=\Wh_\m\dd_\nu:=\dd_\nu/\dd_\nu\mc$.
The lemma says that, for all sufficiently dominant $\la\in\xx$, one has
\beq{ttt2}
\gr_\kk\Wh^\m\Ga\big(\B,\ \oo(\la)\o_{\oo_\B}
(\dd_\nu/\dd_\nu\mc)\big)=
\Ga\big(\ss,\
\oo_\ss(\la)\o_{\oo_\ss}(\tgr_\kk\Wh_\m\dd_\nu)|_\ss)\big).
\eeq

On the other hand, Corollary \ref{Q} yields
$\tgr_\kk\Wh_\m\dd_\nu=\oo_\si$, hence
 $(\tgr_\kk\Wh_\m\dd_\nu)|_\ss=\oo_\ss$.
Thus, using \eqref{ttt1}-\eqref{ttt2}, we get
$$\gr_\kk A_\nu^{\nu+\la}=\gr_\kk\Wh^\m\Ga\big(\B,\qq_\nu^{\nu+\la})=
\Ga(\ss, \
\oo_\ss(\la)\o_{\oo_\ss}\oo_\ss)=\Ga(\ss, \
\oo_\ss(\la)).
$$

The verification of compatibilities with the pairings is left for the reader.
\end{proof}
\section{Noncommutative resolutions of Slodowy slices}\label{nc_sec}
\subsection{Directed algebras}\label{t-alg} Let  $\La$ be a torsion free abelian group,
and  $\La^+\sset\La$ a subsemigroup such that
$\La^+\,\cap\,(-\La^+)=\{0\}.$ Given a pair $\mu,\nu\in\La$,
we write $\mu\seq \nu$ whenever $\mu-\nu\in\La^+.$ This gives
a partial order on $\La$.

A  {\em  directed algebra}  is a vector space
$B:=\bplus_{\mu\seq \nu}\ B_{\mu\nu}$, graded by pairs $\mu,\nu\in\La$
such that $\mu\seq \nu$, and equipped, for each triple
$\mu\seq \nu \seq \la$, with a bilinear multiplication pairing
$B_{\mu\nu}\o B_{\nu\la}\to B_{\mu\la}$. These pairings are required to satisfy,
 for each quadruple $\mu\seq \nu \seq \la\seq \varrho$, a natural associativity
condition, cf. \cite{Mu}.
In the special case where the group $\La=\Z$ is equipped
with the usual order, our definition reduces 
to the notion  of $\Z$-algebra  used in
\cite{GS}, and \cite{Bo}.

In general, for any  directed algebra $B$, 
 the multiplication pairings give each of the spaces
$B_{\mu\mu}$ an associative algebra structure. Similarly,
for each pair $\mu\seq \nu$, the space $B_{\mu\nu}$ acquires 
a $(B_{\mu\mu},B_{\nu\nu})$-bimodule structure. 
Note that, for any $\mu\seq\nu\seq\la$, the mutiplication pairing
descends to a well defined map $B_{\mu\nu}\o_{B_{\nu\nu}} B_{\nu\la}\to B_{\mu\la}$.

\begin{examp}\label{BBB}
Let
$\bb=\bplus_{\la\in\La}\, \bb_\la$ be an ordinary $\La$-graded 
associative algebra. For any pair $\mu\seq\nu$, put
$B_{\mu\nu}:=\bb_{\mu-\nu}$. Then, the {\em bi}graded
space  ${}^{\sharp\!}\bb=\bplus_{\mu\seq \nu}\,
B_{\mu\nu}$ has a natural structure of  directed algebra.
It is called the  directed algebra associated with the graded algebra $\bb$.
\end{examp}

We say that a directed algebra
$B$ is {\em filtered} provided, for each pair $\mu\succ\nu$,
one has
an ascending  $\Z$-filtration $F_\idot B_{\mu,\nu}$, on the corresponding
 component $B_{\mu,\nu}$, and multiplication pairings 
$B_{\mu,\nu}\o B_{\nu,\la}\to B_{\mu,\la}$ respect the filtrations.
There is an associated  graded directed algebra $\gr
B$, with components
$(\gr B)_{\mu\nu}:=\bplus_{i\in\Z}\ F_iB_{\mu\nu}/F_{i-1}B_{\mu\nu}$.

Given a  directed algebra $B=\bplus_{\mu\seq \nu}\,B_{\mu\nu},$ 
one has the notion of an $\La$-graded
$B$-module. Such a module is, by definition,
an $\La$-graded vector space
$M=\bplus_{\nu\in\La}\, M_\nu$ equipped, for each pair
$\mu\seq\nu,$ with an `action map'
$\act_{\mu,\nu}:\ B_{\mu\nu} \o M_\nu\to M_\mu$, that satisfies a
natural associativity condition for each triple
$\mu\seq\nu\seq\la$. Such a module $M$ is said to be finitely
generated if there exists a finite collection
of elements $m_1\in M_{\nu_1},\ldots,m_p\in M_{\nu_p}$
such that, one has
\beq{fin_gen}
M_\mu=\sum_{i=1}^p\ \act_{\mu,\nu_i}(B_{\mu,\nu_i},\,m_i),\qquad
\forall \mu\in\La.
\eeq
We let $\Gr{B}$ denote the category of finitely generated  $\La$-graded left
$B$-modules.

Given a   $\La$-graded
$B$-module $M$, we put $\Spec M:=\{\nu\in\La\mid M_\nu\neq 0\}.$ It is
clear that, for any $M\in \Gr{B}$, there exists a finite
subset  $S\sset\La$ such that we have $\Spec M\sset S+\La^+$.
We say that $M$ is  {\em negligible} if there
exists $\nu\in\La$ such that
$(\nu+\La^+)\cap \Spec M=\emptyset$.
Let $\tail{B}$ be the  full    subcategory of $\Gr{B}$
whose objects are  negligible $B$-modules.

An  directed algebra $B$ is said to be {\em noetherian} if, for each
$\la\in\La$, the algebra $B_\la$ is left noetherian and, moreover,
$B_{\la\nu}$ is a finitely generated left $B_\la$-module, for any $\la\seq\nu.$
It is known, see Boyarchenko \cite[Theorem 4.4(1)]{Bo},
that, for  a noetherian  directed algebra $B$, the category $\Gr{B}$
is  an abelian category, and   $\tail{B}$ is its Serre subcategory.
Thus, one can define $\qgr{B}:=\Gr{B}/\tail{B}$, a Serre quotient category.

Fix $\al\in\La$.
Given a directed algebra $B=\bplus_{\mu\seq \nu}\,B_{\mu\nu},$
resp.  a   $\La$-graded
$B$-module $M=\bplus_{\mu}M_\mu$, put
$B^{\seq \al}=\bplus_{\mu\seq \nu\seq\al}\,B_{\mu\nu}$,
resp. $M^{\seq \al}=\bplus_{\mu\seq\al}M_\mu.$
Thus, $B^{\seq \al}$ may be viewed as a directed
subalgebra of $B$ which has zero homogeneous
components $(B^{\seq \al})_{\mu\nu}$ unless
$\mu\seq \nu\seq\al.$ Similarly,
 $M^{\seq \al}$  may be viewed as a
$B$-submodule in $M$. If $B$ is  noetherian then
 $B^{\seq \al}$ is clearly  noetherian as well.

One easily proves the
following result

\begin{Prop}\label{tail} Assume that the pair $(\La,\La^+)$ satisfies
the following two conditions: 
\vskip 3pt

\npb{\em The semi-group $\La_+$ is finitely generated.}

\npb{\em For any $\mu,\nu\in\La$
the set $(\mu+\La^+)\cap
(\nu+\La^+)$ is nonempty.}\vskip 3pt

\noindent Then, for any  noetherian  directed algebra
$B$
and any  $\al\in\La$, restriction of scalars gives a 
well defined functor
$\Gr{B}\to\Gr{B^{\seq \al}}$. Furthermore, this functor induces an equivalence
$\qgr{B}\iso\qgr{B^{\seq \al}}$.\qed
\end{Prop}
\subsection{Geometric example}
Let $X$ be a quasi-projective algebraic variety, and
write  $\coh X$ for the abelian category of coherent sheaves
on
$X$.
Given  an ample line bundle  $\LL$  on $X$, 
one defines  $\bb(X,\LL):=\bplus_{n\geq 0}\,
\Ga(X,\LL^{\o n})$, a homogeneous coordinate
ring of $X$.
For the corresponding  Proj-scheme, we have
$\Proj\bb(X,\LL)\cong X$. 

 Following Example \ref{BBB} in the special case where $\La=\Z$ and
$\La^+=\Z_{\geq 0}$, we may form the  directed algebra
${}^{\sharp\!}\bb(X,\LL)$ associated with $\bb(X,\LL),$
the latter being viewed as a
$\La$-graded algebra.
Then,
one can construct a natural equivalence of categories
\beq{comm_equiv}
\coh X\cong\qgr{{}^{\sharp\!}\bb(X,\LL)}.
\eeq

We return to the setting  of \S\ref{e}.
Let $\h$ be the Cartan subalgebra for
the semisimple Lie algebra $\g$,  and fix $\xx\sset\h^*$, the
 subsemigroup
of integral dominant weights. For each $\la\in\xx$,
we have the line bundle $\oo_{\ss}(\la)$ on the
 Slodowy variety $\ss$.
The direct sum $\aa(e)=\bplus_{\la\in\xx}\,
\Ga\big(\ss,\,\oo_{\ss}(\la)\big)$
has a natural structure of $\xx$-graded algebra.
For the
corresponding
 multi-homogeneous Proj-scheme, one has
$\ss\cong\Proj\aa(e)$.

Now, in the setting of \S\ref{t-alg}, we  put $\La:=\h^*$, resp
$\La^+:=\xx$,
and let ${}^{\sharp\!}\aa(e)$ be the  directed algebra associated to the
 $\xx$-graded algebra $\aa(e)$.
One can show that ${}^{\sharp\!}\aa(e)$ is a noetherian directed algebra and
the following  multi-homogeneous analogue of the equivalence
\eqref{comm_equiv}
holds
\beq{comm_equiv2}
\coh \ss\cong\qgr{{}^{\sharp\!}\aa(e)}.
\eeq

\subsection{} We are going to
 produce
a family of quantizations of the directed algebra ${}^{\sharp\!}\aa(e)$. 
To this end, we exploit   translation bimodules
$A_\nu^{\nu+\la}$ introduced  in  \S\ref{trans}.

Given  $\nu\in\h^*,$ 
we associate to our  nilpotent element
$e\in\g$  a directed algebra
$\ba(e,\nu):=\bplus_{\mu,\la\in\xx}\ A^{\mu+\la+\nu}_{\la+\nu},$
i.e., using  directed algebra notation, for any $\al\seq\be\seq0$, we put
$A_{\al\be}:=A^{\al+\nu}_{\be+\nu}$.

The resulting directed algebra $\ba(e,\nu)$
has the following 
properties:
\medskip

\vi Each homogeneous component $A^{\mu+\nu}_{\la+\nu}$,
of  $\ba(e,\nu)$, comes equipped with a Kazhdan filtration
$\kk_\idot A^{\mu+\nu}_{\la+\nu}$,
by nonnegative integers, 
such that the multiplication pairings respect the
filtrations. 
\vskip 2pt

\vii For all   $\la\in\xx$, one has $A_{\la\la}=A_{\la+\nu}$ is 
the Premet algebra
associated to the central character that corresponds to the weight $\la+\nu$
via the Harish-Chandra isomorphism.
\medskip

In addition, for a regular dominant $\nu$
and any $\mu,\la\in\xx,$ part (iv), resp.  part (ii),
of Proposition \ref{equiv}, implies the following:
\medskip

\viii The $(A_{\la+\nu},A_\nu)$-bimodule
$A^{\la+\nu}_\nu$ yields an equivalence
$\ \dis\modu{A_\nu}\iso\modu{A_{\la+\nu}}.$
\vskip2pt

\iv The multiplication in $\ba(e,\nu)$ induces an isomorphism
$\ \dis A^{\mu+\la+\nu}_{\la+\nu}\o_{A_{\la+\nu}} A^{\la+\nu}_\nu\iso 
A^{\mu+\la+\nu}_\nu.$
\bigskip

We make $\ba(e,\nu)$ a filtered directed algebra by
using the Kazhdan filtrations on
each homogeneous component $A^{\mu+\la+\nu}_{\la+\nu}$,
cf.  (i) above.

Note that, by construction, for each $\la\in\xx$, the space
$A_{\la0}=A^{\nu+\la}_\nu$ has a structure of right $A_\nu$-module.
Thus, one can introduce a functor
\beq{loc}\loc:\
\modu{A_\nu}\ \to\ \qgr{\ba(e,\nu)},\quad M\
\mapsto\ \loc M:=\bplus_{\la\in\xx}\ A^{\la+\nu}_\nu\o_{A_\nu} M.
\eeq

\begin{Thm}\label{main}
For any $\nu\in\h^*$, we have that $\ba(e,\nu)$ is  a noetherian, filtered
directed algebra.
If $\nu$ is dominant and regular then  $\gr_\kk\ba(e,\nu)\cong
{}^{\sharp\!}\aa(e)^{\seq\nu}$, and the
functor \eqref{loc}  is  an equivalence.
\end{Thm}

The  equivalence of the theorem may be thought of
as some sort of the Beilinson-Bernstein localization theorem
for Slodowy slices, cf. Introduction to \cite{GS} for more discussions
concerning this analogy. 

A different approach
to a  result
closely related to Theorem \ref{main} was also proposed by Losev in an
unpublished manuscript.

\begin{proof}[Proof of Theorem \ref{main}] 
The noetherian property is immediate from Proposition \ref{equiv}, and
the isomorphism
 $\gr_\kk\ba(e,\nu)\cong {}^{\sharp\!}\aa(e)^{\seq\nu}$
follows from  Proposition \ref{equiv}(iii).

The equivalence statement
in the theorem is a consequence of
properties (iii)-(iv)
above, thanks to
 Gordon and Stafford \cite[Lemma 5.5]{GS}, and  Boyarchenko 
 \cite[Theorem 4.4]{Bo}.
\end{proof}

\begin{rem}\label{Ug} Theorem  \ref{main} applies also to enveloping
algebras.
This corresponds, at least formally, to the case of the zero nilpotent element
although the case $e=0$ was not considered in this paper.
Specifically,
introduce a filtered directed algebra
$\UU(\nu):=\bplus_{\mu,\la\in\xx}\ U^{\mu+\la+\nu}_{\la+\nu},$
where $U^{\nu+\la}_\nu:=\Ga(\B, \ \oo(\la)\o_{\oo_\B}\dd_\nu).$
Then, arguing as in the proof of Theorem \ref{main},
one obtains, for  any dominant regular $\nu$, an equivalence
$\modu{\U_\nu}\ \iso\ \qgr{\UU(\nu)}$.
\end{rem}
\begin{rem}\label{dom} One can show that for a dominant, but not necessarily
 regular $\nu$, the functor $M\mto\loc M$ is exact and,
moreover, one has $\ M\neq 0\; \Rightarrow\; \loc M\neq0.$
\end{rem}

\subsection{Characteristic varieties for 
$\ba(e,\nu)$-modules}\label{charvar_dir}
Fix $\La$ as in \S\ref{t-alg} and let $B$ be a filtered directed algebra. A
$\Lambda$-graded $B$-module $M=\bplus_{\nu\in\Lambda}\ M_\nu$ is
said to be  filtered provided
 each of the
spaces $M_\nu$ is equipped with an
ascending  $\Z$-filtration $F_\idot M_\nu$ such that,
for any $i,j\in\Z$ and $\mu,\nu\in\Lambda$,
we have $F_iB_{\mu\nu}\cdot F_jM_\nu\sset F_{i+j}M_\mu$.
In such a case, there is a well defined associated
graded $\Lambda$-module $\gr M$, over $\gr B$.
Thus  $\gr M$ has an additional $\Z$-grading.

A filtration on $M$ is called good if $\gr M$ is a finitely
generated $\gr B$-module in the sense of \eqref{fin_gen}.

Now let $(\La,\La^+)=(\h^*,\xx)$. This pair satisfies
the conditions of Proposition \ref{tail}. Thus,
for any  $\nu\in\h^*$,
combining the proposition and 
\eqref{comm_equiv2}, we get a category equivalence
$\Phi:\ \qgr{{}^{\sharp\!}\aa(e)^{\seq\nu}}\iso\coh \ss.$

Assume first that  $\nu\in\h^*$ is a  dominant and regular
weight. Then, we have  $\gr_\kk\ba(e,\nu)\cong {}^{\sharp\!}\aa(e)^{\seq \nu},$ 
by Theorem 
\ref{main}. Therefore, for any  $\Lambda$-graded $\ba(e,\nu)$-module $M$
equipped
with a good filtration,  we may
view $\gr M$ as a finitely generated graded
${}^{\sharp\!}\aa(e)^{\seq \nu}$-module and put
 $\qqgr M:=\Phi(\gr M)$. This is a 
coherent sheaf on $\ss$. The additional
 $\Z$-grading on $\gr M$ gives that sheaf a
$\C^\times$-equivariant structure.
Thus, $\var M:=\supp(\qqgr M)$ is a $\bullet$-stable closed
algebraic subset in $\ss$.

Now, let $\nu$ be an arbitrary, not necessarily
 dominant and regular,
weight. Then, we find a sufficiently dominant
integral weight $\mu$ such that
$\nu+\mu$ is  a  dominant and regular
weight. Given a  $\Lambda$-graded $\ba(e,\nu)$-module $M$,  we may
view $M^{\seq \nu+\mu}$ as
 a  $\Lambda$-graded $\ba(e,\nu+\mu)$-module.
A good filtration on $M$ induces one on
$M^{\seq \nu+\mu}$, and we have
$\gr(M^{\seq \nu+\mu})=(\gr M)^{\seq \nu+\mu}$.
Thus, one may apply all the
above constructions to the  $\ba(e,\nu+\mu)$-module $M^{\seq \nu+\mu}$.
This way, one defines the set $\var M$
in the general case where $\nu$ is an  arbitrary weight.

One has the following standard result

\begin{Prop}\label{prop_qqgr} 
\vi The set $\var M$ is a coisotropic subset in
$\ss$
which is independent of the choice of a good filtration on $M$.

\vii For any finitely generated $A_\nu$-module $N$, we have
$\pi(\var\loc N)\sset\var N$.

\viii Let $\nu$ be a dominant and regular weight and
 $\cv\in
{(\dd_{\nu},\mc)\mbox{-}\operatorname{mod}}$.
Then, a good filtration on $\cv$
induces a good filtration on
$\loc{\Wh^\m\Ga(\B, \cv)}$ such that one has
$$\tgr_\kk\cv|_\ss\cong \qqgr(\loc{\Wh^\m\Ga(\B, \cv)})
\qquad\qquad(\text{\em isomorphism in }\ \coh\ss).$$
\end{Prop}

\begin{proof}[Sketch of Proof] The coisotropicness statement
in part (i) is a special case of
Gabber's "integrability of characteristics" result \cite{Ga}.
Part (ii) is an analogous to a well known result of
Borho-Brylinski (cf. \cite{BoBr}, proof of Proposition 4.3).
A key point is that the map $\pi: \ss \to\sn$ is proper. 

To prove part (iii), let $\la\in\xx$. Lemma \ref{trans2} yields
an isomorphism
$A_\nu^{\nu+\la}\o_{A_\nu}\Wh^\m\Ga(\B,\cv)=
\Wh^\m\Ga(\B,\,\oo(\la)\o_{\oo_\B}\cv)$.
Thus, by  definition of the functor
$\loc$, we obtain
\beq{qqqgr}
\qqgr(\loc \Wh^\m\Ga(\B,\cv))=\qqgr\left(\bplus_{\la\in\xx}\ 
\Wh^\m\Ga(\B,\,\oo(\la)\o_{\oo_\B}\cv)\right).
\eeq

Recall  the equivalence $\Phi:\ \qgr{{}^{\sharp\!}\aa(e)^{\seq\nu}}\iso
\coh\ss$  considered earlier in this subsection.
The object on the right hand side of \eqref{qqqgr}
equals $\Phi\big(\bplus_{\la\in\xx}\ 
\tgr\Wh^\m\Ga(\B,\,\oo(\la)\o_{\oo_\B}\cv)\big),$
by definition of the functor $\qqgr$.
Further, for $\la$ sufficiently dominant,
we have an isomorphism
$\tgr\Wh^\m\Ga(\B,\,\oo(\la)\o_{\oo_\B}\cv)=
\Ga\big(\ss, \ \oo_\ss(\la)\o_{\oo_\ss}(\tgr_\kk\cv|_\ss)\big),
$
by Lemma \ref{filt}.
Therefore, we deduce an isomorphism
\beq{ggg3}\Phi\left(\bplus_{\la\in\xx}\ 
\tgr\Wh^\m\Ga(\B,\,\oo(\la)\o_{\oo_\B}\cv)\right)
=
\Phi\left(\bplus_{\la\in\xx}\
\Ga\big(\ss, \ \oo_\ss(\la)\o_{\oo_\ss}(\tgr_\kk\cv|_\ss)\big)\right).
\eeq

In general, let  ${\mathcal F}$ be an arbitrary
 coherent sheaf on $\ss$.
By  definition of
the equivalence \eqref{comm_equiv2}, in $\coh\ss$,
one has an isomorphism
$\Phi\left(\bplus_{\la\in\xx}\
\Ga(\ss, \ \oo_\ss(\la)\o_{\oo_\ss}{\mathcal F})\right)
={\mathcal F}.$
Applying this observation to the sheaf
${\mathcal F}:=\tgr_\kk\cv|_\ss,$ and using  isomorphisms
\eqref{qqqgr}-\eqref{ggg3}, we deduce
part (iii) of Proposition \ref{prop_qqgr}.
\end{proof}

To proceed further, we introduce the following terminology.
A coherent sheaf $\mathcal F$, on  a reduced scheme $X$, is said to be
reduced if the annihilator of  $\mathcal F$ is a radical ideal
in $\oo_X$.

\begin{defn}\label{rs_def} An object $V\in \qgr{\ba(e,\nu)}$
is said to have {\textsf{regular singularities}} if there exists
a representative $M\in \Gr{\ba(e,\nu)}$,
of $V$,
and a good filtration on $M$ such that
the corresponding  sheaf $\qqgr M\in \coh\ss$
is reduced and, moreover,
$\supp(\qqgr M)$ is a Lagrangian
subvariety in $\ss$.
\end{defn}

\subsection{Definition of Harish-Chandra $(A_{c'},A_c)$-bimodules}\label{strong}
We fix a pair of weights $\be,\gamma\in\h^*$ and specialize the general setting at the begginning of
section \ref{charvar_dir} to
$B=\ba(e,\be)\o\ba(e,\gamma)^\op$, a filtered
$\Lambda\times\Lambda$-graded directed algebra.
By Theorem \ref{main}, we have
$\gr_\kk(\ba(e,\be)\o\ba(e,\gamma)^\op)={}^{\sharp\!}(\aa(e)
\o\aa(e))^{\seq \be\times\gamma}$.
Hence, there is a category equivalence
$\qgr{\gr_\kk(\ba(e,\be)\o\ba(e,\gamma)^\op)}\cong\coh(\ss\times\ss).$
Therefore, associated with  any $\Lambda\times\Lambda$-graded 
$(\ba(e,\be),\ba(e,\gamma))$-bimodule
$M$ with a good filtration, there is a
$\C^\times$-equivariant coherent sheaf
$\qqgr M$, on $\ss\times\ss$.

One proves the following  analogue of Theorem \ref{main} for bimodules.
\begin{Prop}\label{main_bim}
For dominant  regular $\be,\gamma\in\h^*$,  there is  a canonical equivalence
$$
M\
\longmapsto\ \loc M:=\bplus_{\la,\mu\in\xx}\ (A^{\la+\be}_\be\o_{A_\be}
M\o_{A_\gamma}
A^{\mu+\gamma}_\mu),$$
between the category $\modu{(A_\be\o A_\gamma^\op)},$
of finitely generated $(A_\be,A_\gamma)$-bimodules,
and category $\qgr{\ba(e,\be)\o \ba(e,\gamma)^\op}$.\qed
\end{Prop}

Next, recall the Steinberg variety $\ZZ$ and observe that
 $\ZZ\cap
(\ss\times\ss)$
is a Lagrangian subvariety in $\ss\times\ss$.
An analogue of Proposition \ref{prop_qqgr} for bimodules yields the
following
result.

\begin{Prop}\label{qqgr_bimod} Let $\be,\gamma\in\h^*$ be dominant
weights. Then,

\vi For any  \whc\;
$(A_\be,A_\gamma)$-bimodule $N$, we have
$\var(\loc N)\sset\ZZ\cap (\ss\times\ss)$.

\vii Let $\be$ and $\gamma$ be dominant, and let $\cv$ be a
$G$-monodromic $\dd_{\be,\gamma}$-module. Then,  a good
filtration on $\cv$
induces a good filtration on
$\loc{\Wh^\m_\m\Ga(\B\times\B, \cv)}$
 such that one has
$$\tgr_\kk\cv|_{\ss\times\ss}\ \cong\ \qqgr\loc{\Wh^\m_\m\Ga(\B\times\B, \cv)}.\eqno\Box$$
\end{Prop}
In part (ii) above, we have abused the notation $\Wh^\m_\m$ and,
given a {\em left} $(\U_\be\o \U_\gamma)$-module $N$,
write $\Wh^\m_\m N$ for $(\mc\o 1)$-invariants in $N/(1\o\mc) N.$

Transition from left $\dd_{\be,\gamma}$-modules to
$(\ug,\ug)$-bimodules can be carried out as explained
eg. in \cite[\S5]{BG}, esp. Lemma 5.4 and formula (5.5). 
\medskip

\begin{proof}[\textsf{Proof of Theorem \ref{adj}(i)}] 
Fix a \whc\; $(A_{c'},A_c)$-bimodule $N$ and choose a good filtration on $\loc N$.
We know that  $\var(\loc 
N)\sset\ZZ\cap (\ss\times\ss),$ by the proposition above,
and that $\ZZ\cap
(\ss\times\ss)$
is a Lagrangian variety.
Therefore, Gabber's theorem (cf. Proposition \ref{prop_qqgr}(i)), implies
that
any irreducible component of 
the scheme $\supp(\loc N)$ is Lagrangian, moreover, it is  an irreducible component
of $\ZZ\cap (\ss\times\ss)$. 

We define
$CC( \loc N)$, the {\em characteristic cycle} of $\loc N$,
to be a Lagrangian cycle in
$\ss\times\ss$ equal to the linear combination of the
 irreducible components of 
the scheme $\supp(\loc N)$, counted with multiplicities.
A standard argument shows that the cycle $CC( \loc N)$ is independent of
the choice of good filtration.
Further, an analogue of Remark \ref{dom} for bimodules
implies the exactness of the functor $N \mto\loc N$.
It follows  that the assignment
$N\mto CC(\loc N)$ is additive on
short exact sequences.

Now, let $N=N^0\supset N^1\supset N^2\supset\ldots$ be
a descending chain of sub-bimodules of $N$ such that $N^i/N^{i+1}\neq 0$
for any $i$. Thus, we have
$\loc N^i/\loc N^{i+1}=\loc{N^i/N^{i+1}}\neq 0$, cf.  Remark \ref{dom}.
It follows that the total multiplicity in the
cycle $CC(\loc N^i),\ i=0,1,\ldots,$ gives a strictly decreasing sequence
of natural numbers. Therefore, this sequence terminates.
We conclude  that  $N^i=0$ for $i\gg0,$ hence $N$ has finite length.
\end{proof}

Further, it is clear that Definition \ref{rs_def} of `regular singularities'
may be  adapted to objects of $\qgr{\ba(e,\be)\o\ba(e,\gamma)^\op}$
in an obvious way.
The following definition is motivated by Corollary \ref{rs_cor},
to be discussed in  section \ref{rs_sec} below.

\begin{defn}\label{hc_def} 
A weak Harish-Chandra $(A_\mu,A_\nu)$-bimodule $N$ is called a Harish-Chandra bimodule
if  $\loc N\in\qgr{\ba(e,\mu)\o\ba(e,\nu)^\op},$
the object corresponding to $N$ via the equivalence of Proposition
\ref{main_bim},
has regular singularities.
\end{defn}

Part (ii) of Proposition \ref{qqgr_bimod}, combined with
Corollary \ref{flat}(iii) and with
Proposition \ref{RS} below, yields the following result.

\begin{Prop}\label{HCprop} If $K\in\hc(\U_c,\U_{c'})$ then
$\Wh^\m_\m K$ is a  Harish-Chandra $(A_\mu,A_\nu)$-bimo-dule;
furthermore, we have $\var(\loc{\Wh^\m_\m K})=\var K\cap (\ss\times\ss).$\qed
\end{Prop}

\subsection{Harish-Chandra $\ug$-bimodules vs regular singularities}\label{rs_sec}

\begin{Prop}\label{RS} Let $\mu,\nu\in\h^*$ be  dominant regular weights and let
$\mm$ be a $\dd_{\mu,\nu}$-module such that,
set-theoretically, one has $\var\mm\sset\ZZ$.
Then, the following conditions are equivalent

\vi The $\dd$-module $\mm$ has regular singularities in the sense of \cite{KK};

\vii There exists a good filtration on $\mm$ such that   $\tgr\mm$,
the associated graded sheaf, is
reduced.

\viii There exists a good filtration on $\Ga(\BB,\mm)$ such
that $\gr\Ga(\BB,\mm)$ is a symmetric $(\C[\g^*],\C[\g^*])$-bimodule,
i.e., such that one has $a\cdot m=m\cdot a,\ \forall a\in \C[\g^*],
m\in\gr\Ga(\BB,\mm)$.

\iv The $\g$-diagonal action on $\Ga(\BB,\mm)$ is 
locally finite.
\end{Prop}




\begin{proof} The implication
(i)$\en\Rightarrow\en$(ii) is a special case of a  general
result of Kashiwara-Kawai, cf. also \cite[Definition 5.2]{Kas}.
 The corresponding filtration was defined
in  \cite[Theorem 5.1.6]{KK}.

To prove (ii)$\en\Rightarrow\en$(iii) let $F_\idot\mm$
be  a good
filtration on $\mm$ such that
(ii) holds. 
Any element $x\in\g$ gives rise to a linear
function $\g^*\times\g^*\to\C,\ (\alpha,\be)\mto \langle \alpha+\be, x\rangle.$ Let
$\wt x$ denote the pull-back of  that  function
via the
map $\pi\times(-\pi):\
T^*\B\times T^*\B\to\N\times\N,\ u\times v\mto \pi(u)\times (-\pi(v))$.
Thus, we have $\ZZ=[\pi\times(-\pi)]\inv(\N_\Delta),$ 
where $\N_\Delta\sset\N\times\N$ is the diagonal. Clearly, the function $\wt x$ vanishes on
$\ZZ$. Hence,
we have $\wt x\cdot \tgr\mm=0.$

Observe next that $\wt x$ is a homogeneous function, of degree 1,
along the fibers of the cotangent bundle
$T^*(\BB)$. Thus, the equation
$\wt x\cdot \tgr\mm=0$ implies that, for any
$j\in\Z$, the good filtration on $\mm$
satisfies $x(F_j\mm)\sset F_j\mm$,
where $x(-)$ stands for the $\g$-diagonal action of $x$ on
$\mm$. It follows that the filtration on $\Ga(\BB, \mm)$
defined, for any $j\in\Z$, by
$\Ga_j(\BB, \mm):=\Ga(\BB, F_j\mm)$ makes
$\gr\Ga(\BB, \mm)$ a symmetric $(\C[\g^*],\C[\g^*])$-bimodule.

To prove the implication (iii)$\en\Rightarrow\en$(iv),
let $\Ga_\idot(\BB, \mm)$ be an arbitrary good filtration
on the bimodule $\Ga(\BB, \mm)$ that makes $\gr\Ga(\BB, \mm)$ a
symmetric 
$(\C[\g^*],\C[\g^*])$-bimodule.
The filtration being good, the vector space $\Ga_j(\BB, \mm)$
is finite dimensional for each  $j\in\Z$.
 The symmetry of the bimodule $\gr\Ga(\BB, \mm)$
implies, by an easy induction on $j$, that the vector space $\Ga_j(\BB,
\mm)$ is
stable under the $\g$-diagonal action.
We conclude that
the  $\g$-diagonal action on  $\Ga(\BB, \mm)=\cup_j\ \Ga_j(\BB, \mm)$
is locally finite.

To complete the proof of the theorem,
let $\mm$ be a $\dd$-module such that (iv) holds
and put $M:=\Ga(\BB,\mm)$. It is clear
that the $\g$-diagonal action on $M$ can be exponentiated
to an algebraic $G^{sc}$-action. It follows that $\dd_{\mu,\nu}\o_{\U_\mu\o\U_\nu}M$ is
a $G^{sc}$-equivariant $\dd$-module. But, 
$\BB$ being a projective  variety with
finitely many 
$G^{sc}$-diagonal orbits,
any  $G^{sc}$-equivariant  $\dd$-module on $\BB$ has
regular singularities, see eg. \cite[Theorem 11.6.1]{HTT}. 
Finally, by the Beilinson-Bernstein theorem,
we have $\mm=\dd_{\mu,\nu}\o_{\U_\mu\o\U_\nu}M$,
and the implication
(iv)$\en\Rightarrow\en$(i) follows.
\end{proof}

Using an analogue of the equivalence of Remark \ref{Ug}
for $(\U_\mu,\U_\nu)$-bimodules, one can reformulate
Proposition \ref{RS} as follows

\begin{Cor}\label{rs_cor} Let $\mu,\nu\in\h^*$ be  dominant regular weights
and let $M$ be a \whc\;
$(\U_\mu,\U_\nu)$-bimodule.
Then, the following conditions are equivalent

\vi The adjoint $\g$-action on $M$ is locally finite,
i.e.,
$M\in\hc(\U_\mu,\U_\nu)$ is a Harish-Chandra bimodule
in the sense of \S\ref{hccat}.

\vii The object  $\loc M\in \qgr{\UU(\mu)\o\UU(\nu)^\op}$
 has regular singularities.\qed
\end{Cor}

\small{
\bibliographystyle{plain}

}

\end{document}